\documentclass[12pt, twoside]{article}
\usepackage{times}
\usepackage{bm}
\usepackage{graphics}
\usepackage{theorem}
\usepackage{graphicx}
\usepackage{amsmath}
\usepackage{latexsym}
\usepackage{amssymb,mathrsfs}
\usepackage[flushmargin]{footmisc}

\theorembodyfont{\itshape}
\newtheorem{thm}{Theorem}[section]
\newtheorem{lem}{Lemma}[section]
\newtheorem{prop}{Proposition}[section]

\theorembodyfont{\rmfamily}
\newtheorem{defn}{Definition}[section]{\bf}{\rm}
\newtheorem{assumpt}{Assumption}[section]{\bf}{\rm}

\newtheorem{rem}{Remark}[section]{\itshape}{\rmfamily}

\newenvironment{proof}{\noindent{\it Proof.~~}}{\qed\medskip}

\makeatletter
\def\eqnarray{\stepcounter{equation}\let\@currentlabel=\theequation
\global\@eqnswtrue
\global\@eqcnt\z@\tabskip\@centering\let\\=\@eqncr
$$\halign to \displaywidth\bgroup\@eqnsel\hskip\@centering
  $\displaystyle\tabskip\z@{##}$&\global\@eqcnt\@ne 
  \hfil$\;{##}\;$\hfil
  &\global\@eqcnt\tw@ $\displaystyle\tabskip\z@{##}$\hfil 
   \tabskip\@centering&\llap{##}\tabskip\z@\cr}
\makeatother
\makeatletter
    \renewcommand{\theequation}{%
    \thesection.\arabic{equation}}
    \@addtoreset{equation}{section}
  \makeatother

\newcommand{\AAP}{{\it Adv.\ Appl.\ Prob.}}
\newcommand{\JAP}{{\it J.\ Appl.\ Prob.}}
\newcommand{\SPA}{{\it Stoch.\ Proc.\ Appl.}}
\newcommand{\dm}{\displaystyle}
\newcommand{\qed}{\hspace*{\fill}$\Box$}
\newcommand{\vc}{\bm}
\def\svc#1{\mbox{\boldmath $\scriptstyle #1$}}

\def\trunc#1{{}_{(n)}#1}

\newcommand{\sfBI}{\mathsf{BI}}
\newcommand{\sfBM}{\mathsf{BM}}
\newcommand{\EE}{\mathsf{E}}
\newcommand{\PP}{\mathsf{P}}

\newcommand{\calF}{\mathcal{F}}
\newcommand{\calG}{\mathcal{G}}

\newcommand{\calI}{\mathcal{I}}

\newcommand{\bbD}{\mathbb{D}}
\newcommand{\bbF}{\mathbb{F}}

\newcommand{\bbN}{\mathbb{N}}

\newcommand{\bbZ}{\mathbb{Z}}

\newcommand{\Mod}{\mathrm{mod}}

\newcommand{\rmT}{{\rm T}}

\newcommand{\dd}[1]{\if#11 1\!\!1 
\else {\if#1C I\!\!\!C
\else {\if#1G I\!\!\!G 
\else {\if#1J J\!\!\!J 
\else {\if#1S S\!\!\!S
\else {\if#1Z Z\!\!\!Z
\else {\if#1Q O\!\!\!\!Q
\else I\!\!#1
\fi} 
\fi}
\fi}
\fi} 
\fi} 
\fi} 
\fi} 

\makeatother

\begin{document}\thispagestyle{plain} 

\hfill
{\small Last update date: \today}

{\Large{\bf
\begin{center}
Error Bounds for Augmented Truncations of Discrete-Time Block-Monotone 
Markov Chains under Geometric Drift Conditions%
\footnote[1]{This paper has been accepted for publication in 
{\it Advances in Applied Probability}. }
%
%
\end{center}
}
}

\begin{center}
{
Hiroyuki Masuyama%
\footnote[2]{E-mail: masuyama@sys.i.kyoto-u.ac.jp}
}

\medskip

{\small
Department of Systems
Science, Graduate School of Informatics, Kyoto University\\
Kyoto 606-8501, Japan
}

\bigskip
\medskip

{\small
\textbf{Abstract}

\medskip

\begin{tabular}{p{0.85\textwidth}}
This paper studies the augmented truncation of discrete-time
block-monotone Markov chains under geometric drift conditions. We
first present a bound for the total variation distance between the
stationary distributions of an original Markov chain and its augmented
truncation. We also obtain such error bounds for more general cases
where an original Markov chain itself is not necessarily
block-monotone but is block-wise dominated by a block-monotone Markov
chain.  Finally we discuss the application of our results to
GI/G/1-type Markov chains.
\end{tabular}
}
\end{center}

\begin{center}
\begin{tabular}{p{0.90\textwidth}}
{\small

\medskip

}
\end{tabular}

\end{center}

\section{Introduction}\label{introduction}

Various semi-Markovian queues and their state-dependent extensions can
be analyzed through block-structured Markov chains characterized by an
infinite number of block matrices, such as level-dependent
quasi-birth-and-death processes (LD-QBDs), M/G/1-, GI/M/1- and
GI/G/1-type Markov chains (see, e.g., \cite{He14-book}).

For LD-QBDs, there exist some numerical procedures based on the
$RG$-factorization, though their implementation requires the
truncation of the infinite sequence of block matrices in a heuristic
way \cite{Baum10,Brig95,Tuan10}. Such ``truncation in implementation"
is also necessary for {\it level-independent} M/G/1- and GI/M/1-type
Markov chains (see, e.g., Section 4 in \cite{Taki94-PE}) and thus for
GI/G/1-type ones. As far as we know, there is no study on the
computation of the stationary distributions of {\it level-dependent}
M/G/1- and GI/M/1-type Markov chains and more general ones. For these
Markov chains, the $RG$-factorization method does not seem effective
in developing numerical procedures with {\it good} properties, such as
space- and time-saving and guarantee of accuracy, because the
resulting expression of the stationary distribution is characterized
by an infinite number of $R$- and $G$-matrices \cite{Zhao98}. As for
the transient distribution, Masuyama and Takine~\cite{Masu05} propose
a stable and accuracy-guaranteed algorithm based on the uniformization
technique (see, e.g., \cite{Tijm03}).

As mentioned above, it is challenging to develop a numerical procedure
for computing the stationary distributions of block-structured Markov
chains characterized by an infinite number of block matrices.  A
practical and simple solution for this problem is to truncate the
transition probability matrix so that it is of a finite dimension. The
stationary distribution of the resulting finite Markov chain can be
computed by a general purpose algorithm, in principle. However, the
obtained stationary distribution includes error caused by truncating
the original transition probability matrix. Therefore from a practical
point of view, it is significant to estimate ``truncation error".

Tweedie \cite{Twee98} and Liu \cite{Liu10} study the estimation of
error caused by truncating (stochastically) monotone Markov chains
(see, e.g., \cite{Dale68}). Tweedie \cite{Twee98} presents error
bounds for the last-column-augmented truncation of a monotone Markov
chain with geometric ergodicity.  The last-column-augmented truncation
is constructed by augmenting the last column of the {\it northwest
  corner truncation} of a transition probability matrix so that the
resulting finite matrix is stochastic. On the other hand, Liu
\cite{Liu10} assumes that a monotone Markov chain is subgeometrically
ergodic and then derives error bounds for the last-column-augmented
truncation.

Unfortunately, block-structured Markov chains are not monotone in
general. Li and Zhao \cite{Li00} extend the notion of monotonicity to
block-structured Markov chains. The new notion is called
``(stochastic) block-monotonicity". Block-monotone Markov chains
(BMMCs) arise from queues in Markovian environments, such as queues
with batch Markovian arrival process (BMAP) \cite{Luca91}. Li and Zhao
\cite{Li00} prove that if an original Markov chain is block-monotone,
then the stationary distributions of its augmented truncations
converge to that of the original Markov chain, which motivates this
study.

In what follows, we give an overview of Li and Zhao~\cite{Li00}'s
work. To this end, we introduce some notations. Let $\bbN =
\{1,2,3,\dots\}$. Let $\bbZ_+^{\leqslant n} = \{0,1,\dots,n\}$ for $n
\in \bbN$ and $\bbZ_+^{\leqslant \infty} := \bbZ_+ =
\{0,1,2,\dots\}$. Further let $\bbF^{\leqslant n} = \bbZ_+^{\leqslant
  n} \times \bbD$ for $n \in \overline{\bbN}:=\bbN \cup \{\infty\}$,
where $\bbD = \{1,2,\dots,d\}$. For simplicity, we write $\bbF$ for
$\bbF^{\leqslant \infty}$.

The following is the definition of block monotonicity for stochastic
matrices.
\begin{defn}[Definition 2.5 in \cite{Li00}]\label{defn-BM}
For any $n \in \overline{\bbN}$, a stochastic matrix
$\vc{S}=(s(k,i;l,j))_{(k,i),(l,j)\in\bbF^{\leqslant n}}$ and a Markov
chain characterized by $\vc{S}$ are said to be (stochastically)
block-monotone with block size $d$ if for all $k \in \bbZ_+^{\leqslant
  n-1}$ and $l \in \bbZ_+^{\leqslant n}$,
\[
\sum_{m=l}^n s(k,i;m,j)
\le \sum_{m=l}^n s(k+1,i;m,j),\qquad i,j \in \bbD.
\]
We denote by $\sfBM_d$ the set of block-monotone stochastic matrices
with block size $d$.
\end{defn}

Let $\vc{P}=(p(k,i;l,j))_{(k,i),(l,j)\in\bbF}$ denote a stochastic
matrix.  Let $\{(X_{\nu},J_{\nu});\nu\in\bbZ_+\}$ denote a bivariate
Markov chain with state space $\bbF$ and transition probability matrix
$\vc{P}$. The following result is obvious from the definition. We thus
omit the proof.
\begin{prop}\label{prop-gamma(i,j)}
If $\vc{P} \in \sfBM_d$, then $\psi(i,j) :=
\sum_{l=0}^{\infty}p(k,i;l,j)$ ($i,j\in\bbD$) is constant with respect
to $k \in \bbZ_+$ and $\{J_{\nu};\nu\in\bbZ_+\}$ is a Markov chain
whose transition probability matrix is given by $\vc{\varPsi} :=
(\psi(i,j))_{i,j\in\bbD}$, i.e., $\psi(i,j) = \PP(J_{\nu+1} = j \mid
J_{\nu}=i)$ for $i, j\in\bbD$.
\end{prop}

Proposition~\ref{prop-gamma(i,j)} implies {\it the pathwise ordered
  property} of BMMCs (see Lemma~\ref{lem1-discrete-ordering}): If
$\vc{P} \in \sfBM_d$, then there exist two BMMCs
$\{(X'_{\nu},J'_{\nu});\nu\in\bbZ_+\}$ and
$\{(X''_{\nu},J''_{\nu});\nu\in\bbZ_+\}$ with transition probability
matrix $\vc{P}$ on a common probability $(\Omega, \calF, \PP)$ such
that $X'_{\nu}\le X''_{\nu}$ and $J'_{\nu} = J''_{\nu}$ for all $\nu
\in \bbN$ if $X_0'\le X_0''$ and $J_0' = J_0''$.

Let
$\trunc{\vc{P}}_{\ast}=(\trunc{p}_{\ast}(k,i;l,j))_{(k,i),(l,j)\in\bbF}$
($n \in \bbN$) denote a stochastic matrix such that for $i,j \in
\bbD$,
\begin{align*}
&&
\trunc{p}_{\ast}(k,i;l,j) &\ge p(k,i;l,j), 
& & k \in \bbZ_+,\ l \in \bbZ_+^{\leqslant n}, &&
\\
&&
\trunc{p}_{\ast}(k,i;l,j) &= 0, 
& & k \in \bbZ_+,\ l \in \bbZ_+\setminus\bbZ_+^{\leqslant n}, &&
\\
&&
\sum_{l=0}^n\trunc{p}_{\ast}(k,i;l,j) &= \sum_{l=0}^{\infty}p(k,i;l,j),
&& k \in \bbZ_+. &&
\end{align*}
The stochastic matrix $\trunc{\vc{P}}_{\ast}$ is called {\it a
  block-augmented first-$n$-block-column truncation} (for short,
block-augmented truncation) of $\vc{P}$. 

\begin{rem}
The block-augmented truncation
$\trunc{\vc{P}}_{\ast}$ can be partitioned as
\begin{equation}
\trunc{\vc{P}}_{\ast}
= \bordermatrix{
               	& \bbF^{\leqslant n} 	&   \bbF \setminus \bbF^{\leqslant n}    
\cr
\bbF^{\leqslant n} 	& \trunc{\vc{P}}_{\ast}^{\leqslant n}	& \vc{O}
\cr
\bbF \setminus \bbF^{\leqslant n} 		& \ast &	\vc{O} 
},
\label{add-eqn-22}
\end{equation}
where $\trunc{\vc{P}}_{\ast}^{\leqslant n}$ is equivalent to the
block-augmented truncation defined in Li and Zhao \cite{Li00}. Our
definition facilitates the algebraic operation for the original
stochastic matrix $\vc{P}$ and its block-augmented truncation
$\trunc{\vc{P}}_{\ast}$ because they are of the same dimension.
\end{rem}

Throughout this paper, unless otherwise stated, we assume that
$\vc{P}$ is irreducible and positive recurrent and then denote its
unique stationary probability vector by $\vc{\pi}=(\pi(k,i))_{(k,i)\in
  \bbF} > \vc{0}$ (see, e.g., Theorem~3.1 in Section~3.1 of
\cite{Brem99}). However, $\trunc{\vc{P}}_{\ast}$ may have more than
one positive recurrent (communication) class in $\bbF^{\leqslant n}$.

Let $\trunc{\vc{\pi}}_{\ast}=(\trunc{\pi}_{\ast}(k,i))_{(k,i)\in\bbF}$
($n \in \bbN$) denote a stationary probability vector of
$\trunc{\vc{P}}_{\ast}$.  Equation (\ref{add-eqn-22}) implies that
$\trunc{\pi}_{\ast}(k,i) = 0$ for all $(k,i)\in
\bbF\setminus\bbF^{\leqslant n}$ (see, e.g., Theorem~1 in Section~I.7
of \cite{Chun67}) and $\trunc{\vc{\pi}}_{\ast}^{\leqslant
  n}:=(\trunc{\pi}_{\ast}(k,i))_{(k,i)\in\bbF^{\leqslant n}}$ is a
solution of $\trunc{\vc{\pi}}_{\ast}^{\leqslant
  n}\trunc{\vc{P}}_{\ast}^{\leqslant
  n}=\trunc{\vc{\pi}}_{\ast}^{\leqslant n}$ and
$\trunc{\vc{\pi}}_{\ast}^{\leqslant n}\vc{e}=1$, where $\vc{e}$
denotes a column vector of ones with an appropriate dimension. It is
also known that if $\vc{P} \in \sfBM_d$, then
$\lim_{n\to\infty}\trunc{\vc{\pi}}_{\ast} = \vc{\pi}$, where the
convergence is element-wise (see Theorem~3.4 in Li and
Zhao~\cite{Li00}).

Let $\trunc{\vc{P}}_n=(\trunc{p}_n(k,i;l,j))_{(k,i),(l,j)\in\bbF}$ ($n
\in \bbN$) denote a block-augmented truncation of $\vc{P}$ such that
for $i,j\in\bbD$,
\begin{equation}
\trunc{p}_n(k,i;l,j)
= 
\left\{
\begin{array}{ll}
p(k,i;l,j), & k \in \bbZ_+,\ l \in \bbZ_+^{\leqslant n-1},
\\
\dm\sum_{m=n}^{\infty} p(k,i;m,j), & k \in \bbZ_+,\ l = n,
\\
0, & \mbox{otherwise},
\end{array}
\right.
\label{def-trunc{p}_n(k,i;l,j)}
\end{equation}
which is called {\it the last-column-block-augmented
  first-$n$-block-column truncation} (for short, the
last-column-block-augmented truncation).  Let
$\trunc{\vc{\pi}}_n=(\trunc{\pi}_n(k,i))_{(k,i)\in\bbF}$ ($n \in
\bbN$) denote a stationary probability vector of $\trunc{\vc{P}}_n$,
where $\trunc{\pi}_n(k,i) = 0$ for all $(k,i)\in
\bbF\setminus\bbF^{\leqslant n}$. We then have the following result.
\begin{prop}[Theorem~3.6 in \cite{Li00}]\label{prop-Li00}
If $\vc{P} \in \sfBM_d$ and $\trunc{\vc{\pi}}_n$ is the unique
stationary distribution of $\trunc{\vc{P}}_n$, then there exists an
infinite increasing sequence $\{n_k \in \bbN;k\in\bbZ_+\}$ such that
for all $k \in \bbZ_+$,
\[
0 \le
\sum_{l=0}^{n_k}\sum_{i\in\bbD}
\left( \trunc{\pi}_n(l,i) - \pi(l,i) \right)
\le 
\sum_{l=0}^{n_k}\sum_{i\in\bbD}
\left( \trunc{\pi}_{\ast}(l,i) - \pi(l,i) \right).
\]
\end{prop}

\medskip

Based on Proposition~\ref{prop-Li00}, Li and Zhao~\cite{Li00} state
that the last-column-block-augmented truncation $\trunc{\vc{P}}_n$ is
the {\it best} approximation to $\vc{P}$ among the block-augmented
truncations of $\vc{P}$, though they do not estimate the distance
between $\trunc{\vc{\pi}}_n$ and $\vc{\pi}$.

In this paper, we consider some cases where $\vc{P}$ satisfies the
geometric drift condition (see Section 15.2.2
in \cite{Meyn09}) but may be periodic. We first assume $\vc{P} \in
\sfBM_d$ and present a bound for the total variation distance between
$\trunc{\vc{\pi}}_n$ and $\vc{\pi}$, which is expressed as follows:
\[
\left\| \trunc{\vc{\pi}}_n - \vc{\pi} \right\| 
:= \sum_{(k,i)\in\bbF} | \trunc{\pi}_n(k,i) - \pi(k,i) | \le C_m(n),
\]
where $C_m$ is some function on $\bbZ_+$ with a supplementary
parameter $m \in \bbN$ such that \break
$\lim_{m\to\infty}\lim_{n\to\infty}C_m(n) = 0$. The bound presented in
this paper is a generalization of that in Tweedie \cite{Twee98} (see
Theorem 4.2 therein).  We also obtain such error bounds for more
general cases where $\vc{P}$ itself is not necessarily block-monotone but is
block-wise dominated by a block-monotone stochastic matrix.

The rest of this paper is divided into four
sections. Section~\ref{sec-preliminaries} provides preliminary results
on block-monotone stochastic matrices. The main result of this paper
is presented in Section~\ref{sec-main-results}, and some extensions
are discussed in Section~\ref{sec-extension}. As an example, these
results are applied to GI/G/1-type Markov chains in
section~\ref{sec-applications}.

\section{Preliminaries}\label{sec-preliminaries}

In this section, we first introduce some definitions and notations,
and then provide some basic results on block-monotone stochastic
matrices.

\subsection{Definitions and notations}

Let $\vc{I}$ denote an identity matrix whose dimension depends on the
context (we may write $\vc{I}_m$ to represent the $m \times m$
identity matrix). For any square matrix $\vc{M}$, let $\vc{M}^0 =
\vc{I}$.  Let $\vc{T}_d$ and $\vc{T}_d^{-1}$ denote
\[
\vc{T}_d
= \left(
\begin{array}{ccccc}
\vc{I}_d & \vc{O} & \vc{O} & \vc{O} & \cdots
\\
\vc{I}_d & \vc{I}_d & \vc{O} & \vc{O} & \cdots
\\
\vc{I}_d & \vc{I}_d & \vc{I}_d & \vc{O} & \cdots
\\
\vc{I}_d & \vc{I}_d & \vc{I}_d & \vc{I}_d & \cdots
\\
\vdots      & \vdots      & \vdots      & \vdots      & \ddots
\end{array}
\right),~~
\vc{T}_d^{-1}
= \left(
\begin{array}{ccccc}
\vc{I}_d & \vc{O} & \vc{O} & \vc{O} & \cdots
\\
-\vc{I}_d & \vc{I}_d & \vc{O} & \vc{O} & \cdots
\\
\vc{O} & -\vc{I}_d & \vc{I}_d & \vc{O} & \cdots
\\
\vc{O} & \vc{O} & -\vc{I}_d & \vc{I}_d & \cdots
\\
\vdots      & \vdots      & \vdots      & \vdots      & \ddots
\end{array}
\right),
\]
where $\vc{T}_d\vc{T}_d^{-1} = \vc{T}_d^{-1}\vc{T}_d = \vc{I}$.  Let
$\vc{T}_d^{\leqslant n}$ ($n \in \overline{\bbN}$) denote the
$|\bbF^{\leqslant n}| \times |\bbF^{\leqslant n}|$ northwest corner
truncation of $\vc{T}_d$, where $|\,\cdot\,|$ denotes set
cardinality. Note that $\vc{T}_d=\vc{T}_d^{\leqslant
  \infty}$ and $(\vc{T}_d^{\leqslant n})^{-1}$ ($n \in
\overline{\bbN}$) is equal to the $|\bbF^{\leqslant n}| \times
|\bbF^{\leqslant n}|$ northwest corner truncation of $\vc{T}_d^{-1}$.

We now introduce the following definitions.
\begin{defn}[Definition 2.1 in \cite{Li00}]\label{defn-BI}
For $n \in \overline{\bbN}$, let
$\vc{f}=(f(k,i))_{(k,i)\in\bbF^{\leqslant n}}$ denote a column vector
with block size $d$.  The vector $\vc{f}$ is said to be
block-increasing if $(\vc{T}_d^{\leqslant n})^{-1}\vc{f} \ge \vc{0}$,
i.e., $f(k,i) \le f(k+1,i)$ for all $(k,i) \in \bbZ_+^{\leqslant n-1}
\times \bbD$. We denote by $\sfBI_d$ the set of block-increasing
column vectors with block size $d$.
\end{defn}

\begin{defn}\label{defn-block-wise-dominance}
For $n \in \overline{\bbN}$, let
$\vc{\mu}=(\mu(k,i))_{(k,i)\in\bbF^{\leqslant n}}$ and
$\vc{\eta}=(\eta(k,i))_{(k,i)\in\bbF^{\leqslant n}}$ denote
probability vectors with block size $d$.  The vector $\vc{\mu}$ is
said to be (stochastically) block-wise dominated by $\vc{\eta}$
(denoted by $\vc{\mu} \prec_d \vc{\eta}$) if
$\vc{\mu}\vc{T}_d^{\leqslant n} \le \vc{\eta}\vc{T}_d^{\leqslant n}$.
\end{defn}

\begin{defn}\label{defn-block-wise-domination}
For $n \in \overline{\bbN}$, let
$\vc{P}_h=(p_h(k,i;l,j))_{(k,i),(l,j)\in\bbF^{\leqslant n}}$ ($h=1,2$)
denote a stochastic matrix with block size $d$. The matrix $\vc{P}_1$
is said to be (stochastically) block-wise dominated by $\vc{P}_2$
(denoted by $\vc{P}_1 \prec_d \vc{P}_2$) if
$\vc{P}_1\vc{T}_d^{\leqslant n} \le \vc{P}_2\vc{T}_d^{\leqslant n}$.
\end{defn}

\begin{rem}\label{prop-3}
The columns of $\vc{T}_d^{\leqslant n}$ are linearly independent
vectors in $\sfBI_d$, and thus every vector $\vc{f} \in \sfBI_d$ is
expressed as a linear combination of columns of $\vc{T}_d^{\leqslant
  n}$. Therefore $\vc{\mu} \prec_d \vc{\eta}$ (resp.~$\vc{P}_1 \prec_d
\vc{P}_2$) if and only if $\vc{\mu}\vc{f} \le \vc{\eta}\vc{f}$
(resp.~$\vc{P}_1\vc{f} \le \vc{P}_2\vc{f}$) for any $\vc{f} \in
\sfBI_d$. According to this equivalence, we can define the block-wise
dominance relation ``$\prec_d$" (see Definitions 2.2 and 2.7 in
\cite{Li00}).
\end{rem}

\subsection{Basic results on block-monotone stochastic matrices}

In this subsection, we present three propositions. The first two of
them hold for any $|\bbF^{\leqslant n}| \times |\bbF^{\leqslant n}|$
($n \in \overline{\bbN}$) stochastic matrix $\vc{S}=(s(k,i;l,j))$ in
$\sfBM_d$. The first proposition is immediate from
Definition~\ref{defn-BM} and thus its proof is omitted. The second one
is an extension of Theorem~1.1 in \cite{Keil77}.
\begin{prop}\label{prop-BM_d}
$\vc{S} \in \sfBM_d$ if and only if $(\vc{T}_d^{\leqslant
    n})^{-1}\vc{S}\vc{T}_d^{\leqslant n} \ge \vc{O}$.
\end{prop}

\begin{prop}\label{prop-2}
The following are equivalent:
\begin{enumerate}
\item $\vc{S} \in \sfBM_d$.
\item $\vc{\mu}\vc{S} \prec_d \vc{\eta}\vc{S}$ for any two probability
  vectors $\vc{\mu}$ and $\vc{\eta}$ such that $\vc{\mu} \prec_d
  \vc{\eta}$.
\item $\vc{S}\vc{f} \in \sfBI_d$ for any $\vc{f} \in \sfBI_d$.
\end{enumerate}
\end{prop}

\begin{rem}
The equivalence of (a) and (c) is shown in Theorem~3.8 in \cite{Li00}.
\end{rem}

\begin{proof}[Proof of Proposition~\ref{prop-2}]
(a) $\Rightarrow$ (b):~~We assume that $\vc{S} \in \sfBM_d$ and
  $\vc{\mu} \prec_d \vc{\eta}$. It then follows from
  Proposition~\ref{prop-BM_d} and
  Definition~\ref{defn-block-wise-dominance} that
  $(\vc{T}_d^{\leqslant n})^{-1}\vc{S}\vc{T}_d^{\leqslant n}\ge
  \vc{O}$ and $\vc{\mu}\vc{T}_d^{\leqslant n}\le
  \vc{\eta}\vc{T}_d^{\leqslant n}$. Thus we have
\[
\vc{\mu}\vc{S}\vc{T}_d^{\leqslant n}
= \vc{\mu}\vc{T}_d^{\leqslant n} 
\cdot (\vc{T}_d^{\leqslant n})^{-1} \vc{S}\vc{T}_d^{\leqslant n} 
\le \vc{\eta}\vc{T}_d^{\leqslant n} 
\cdot (\vc{T}_d^{\leqslant n})^{-1}\vc{S}\vc{T}_d^{\leqslant n}
= \vc{\eta}\vc{S}\vc{T}_d^{\leqslant n},
\]
which shows $\vc{\mu}\vc{S} \prec_d \vc{\eta}\vc{S}$.

(b) $\Rightarrow$ (a):~~For $(k,i) \in \bbF^{\leqslant n}$, let
$\vc{\xi}_{(k,i)}=(\xi_{(k,i)}(l,j))_{(l,j) \in \bbF^{\leqslant n}}$
denote a $1 \times |\bbF^{\leqslant n}|$ unit vector whose $(k,i)$th
element is equal to one. Let $\vc{\eta} = \vc{\xi}_{(k,i)}$ and
$\vc{\mu}=\vc{\xi}_{(k-1,i)}$ for any fixed $(k,i) \in
(\bbZ_+^{\leqslant n}\setminus\{0\}) \times \bbD$. It then follows
that $\vc{\mu} \prec_d \vc{\eta}$ and thus condition~(b) yields
$(\vc{\eta}-\vc{\mu})\vc{S}\vc{T}_d^{\leqslant n}\ge \vc{0}$, where
$\vc{\eta} - \vc{\mu}$ is equal to the $(k,i)$th row of
$(\vc{T}_d^{\leqslant n})^{-1}$. Further $\vc{\xi}_{(0,i)}
\vc{S}\vc{T}_d^{\leqslant n}\ge \vc{0}$ ($i\in\bbD$), where
$\vc{\xi}_{(0,i)}$ is equal to the $(0,i)$th row of
$(\vc{T}_d^{\leqslant n})^{-1}$. As a result, we have
$(\vc{T}_d^{\leqslant n})^{-1}\vc{S}\vc{T}_d^{\leqslant n} \ge
\vc{O}$, i.e., $\vc{S} \in \sfBM_d$ (see Proposition~\ref{prop-BM_d}).

(a) $\Rightarrow$ (c):~~According to Definition~\ref{defn-BI},
$(\vc{T}_d^{\leqslant n})^{-1}\vc{f} \ge \vc{0}$ for any $\vc{f} \in
\sfBI_d$. Combining this with $(\vc{T}_d^{\leqslant
  n})^{-1}\vc{S}\vc{T}_d^{\leqslant n}\ge \vc{O}$ (due to
condition~(a)), we obtain
\[
(\vc{T}_d^{\leqslant n})^{-1}\vc{S}\vc{f} =
(\vc{T}_d^{\leqslant n})^{-1}\vc{S}\vc{T}_d^{\leqslant n} \cdot
(\vc{T}_d^{\leqslant n})^{-1}\vc{f} \ge \vc{0},
\]
and thus $\vc{S}\vc{f} \in \sfBI_d$.

(c) $\Rightarrow$ (a):~~Fix $\vc{f}$ to be a column of
$\vc{T}_d^{\leqslant n}$. Since $\vc{f} \in \sfBI_d$, it follows from
condition~(c) that $\vc{S}\vc{f} \in \sfBI_d$, i.e.,
$(\vc{T}_d^{\leqslant n})^{-1}\vc{S}\vc{f} \ge \vc{0}$. Therefore
$(\vc{T}_d^{\leqslant n})^{-1}\vc{S}\vc{T}_d^{\leqslant n}\ge \vc{O}$.
\end{proof}

The last proposition is a fundamental result for any two
$|\bbF^{\leqslant n}| \times |\bbF^{\leqslant n}|$ ($n \in
\overline{\bbN}$) stochastic matrices $\vc{P}_1=(p_1(k,i;l,j))$ and
$\vc{P}_2=(p_2(k,i;l,j))$ such that $\vc{P}_1 \prec_d \vc{P}_2$, which
is an extension of Lemma~1 in \cite{Gibs87}.
\begin{prop}\label{prop-4}
If $\vc{P}_1 \prec_d \vc{P}_2$ and either $\vc{P}_1 \in \sfBM_d$ or
$\vc{P}_2 \in \sfBM_d$, then the following statements hold:
\begin{enumerate}
\item  For all $k \in \bbZ_+^{\leqslant n}$ and $i,j \in \bbD$,
\[
\sum_{l\in\bbZ_+^{\leqslant n}}p_1(k,i;l,j)=\sum_{l\in\bbZ_+^{\leqslant n}}p_2(k,i;l,j),
\quad \mbox{which is constant with respect to $k$}.
\]
\item $\vc{P}_1^m \prec_d \vc{P}_2^m$ for all $m \in \bbN$.
\item Suppose that $\vc{P}_2$ is irreducible. If $\vc{P}_2$ is
  recurrent (resp.~positive recurrent), then $\vc{P}_1$ has exactly
  one recurrent (resp.~positive recurrent) class that includes the
  states $\{(0,i);i\in\bbD\}$, which is reachable from all the other
  states with probability one. Thus if $\vc{P}_2$ is positive
  recurrent, then $\vc{P}_1$ and $\vc{P}_2$ have the unique stationary
  distributions $\vc{\pi}_1$ and $\vc{\pi}_2$, respectively, and
  $\vc{\pi}_1 \prec_d \vc{\pi}_2$.
\end{enumerate}

\end{prop}

\begin{proof}
We consider only the case of $\vc{P}_1 \in \sfBM_d$ because the case
of $\vc{P}_2 \in \sfBM_d$ is discussed in a very similar way. 
We first prove statement (a). It follows from $\vc{P}_1 \in \sfBM_d$
and Proposition~\ref{prop-gamma(i,j)} that
$\sum_{l\in\bbZ_+^{\leqslant n}}p_1(k,i;l,j)$ is constant with respect
to $k$ for each $(i,j) \in\bbD^2$, which is denoted by
$\psi_1(i,j)$. Further from $\vc{P}_1 \prec_d \vc{P}_2$, we have
\begin{equation}
\psi_1(i,j) 
= \sum_{l\in\bbZ_+^{\leqslant n}}p_1(k,i;l,j)
\le \sum_{l\in\bbZ_+^{\leqslant n}}p_2(k,i;l,j),
\qquad k \in \bbZ_+^{\leqslant n},~i,j \in \bbD.
\label{ineqn-gamma_1(i,j)}
\end{equation}
Since $\vc{P}_1$ and $\vc{P}_2$ are stochastic matrices,
$\sum_{j\in\bbD}\psi_1(i,j) =
\sum_{j\in\bbD}\sum_{l\in\bbZ_+^{\leqslant n}}p_2(k,i;l,j)=1$ for all
$(k,i) \in \bbF^{\leqslant n}$. From this and
(\ref{ineqn-gamma_1(i,j)}), we obtain $\psi_1(i,j) =
\sum_{l\in\bbZ_+^{\leqslant n}}p_2(k,i;l,j)$ for all $k \in
\bbZ_+^{\leqslant n}$ and $i,j\in \bbD$.

Next we prove statement (b) by induction.  Suppose that for some
$m \in \bbN$, $\vc{P}_1^m \prec_d \vc{P}_2^m$, i.e.,
$\vc{P}_1^m\vc{T}_d^{\leqslant n} \le \vc{P}_2^m\vc{T}_d^{\leqslant
  n}$ (which is true at least for $m=1$). Combining this with
$(\vc{T}_d^{\leqslant n})^{-1} \vc{P}_1 \vc{T}_d^{\leqslant n} \ge
\vc{O}$ (due to $\vc{P}_1 \in \sfBM_d$) yields
\begin{eqnarray*}
\vc{P}_1^{m+1}\vc{T}_d^{\leqslant n}
&=& \vc{P}_1^m \vc{T}_d^{\leqslant n} 
\cdot (\vc{T}_d^{\leqslant n})^{-1} \vc{P}_1 \vc{T}_d^{\leqslant n}
\nonumber
\\
&\le& \vc{P}_2^m \vc{T}_d^{\leqslant n} 
\cdot (\vc{T}_d^{\leqslant n})^{-1} \vc{P}_1 \vc{T}_d^{\leqslant n}
= \vc{P}_2^m \cdot  \vc{P}_1\vc{T}_d^{\leqslant n}
\nonumber
\\
&\le& \vc{P}_2^m \cdot  \vc{P}_2\vc{T}_d^{\leqslant n}
= \vc{P}_2^{m+1} \vc{T}_d^{\leqslant n},
\end{eqnarray*}
and thus $\vc{P}_1^{m+1} \prec_d \vc{P}_2^{m+1}$. Therefore statement
(b) is true.

Finally we prove statement (c). Note that there exist two Markov
chains characterized by $\vc{P}_1$ and $\vc{P}_2$, called Markov
chains 1 and 2, which are pathwise ordered by the block-wise dominance
of $\vc{P}_2$ over $\vc{P}_1$ (see
Lemma~\ref{lem2-discrete-ordering}). Since $\vc{P}_2$ is irreducible
and recurrent, Markov chain 2 and thus Markov chain 1 can reach any
state $(0,i)$ ($i\in\bbD$) from all the states in the state space
$\bbF^{\leqslant n}$ with probability one and the mean first passage
time to each state $(0,i)$ ($i\in\bbD$) is finite if $\vc{P}_2$ is
positive recurrent. These facts show that the first part of statement
(c) holds.  Finally we prove $\vc{\pi}_1 \prec_d \vc{\pi}_2$. Note
here that $(\vc{I}+\vc{P}_h)/2$ ($h=1,2$) is aperiodic and has the
same stationary distribution as that of $\vc{P}_h$. Thus we assume
without loss of generality that $\vc{P}_h$ ($h=1,2$) is aperiodic.  It
then follows from statement (b) and the dominated convergence theorem
that $\vc{e}\vc{\pi}_1\vc{T}_d^{\leqslant n} \le
\vc{e}\vc{\pi}_2\vc{T}_d^{\leqslant n}$ (see Theorem~4 in Section I.6
of \cite{Chun67}) and thus $\vc{\pi}_1 \vc{T}_d^{\leqslant n} \le
\vc{\pi}_2\vc{T}_d^{\leqslant n}$.
\end{proof}

\section{Main result}\label{sec-main-results}

This section presents a bound for $\| \trunc{\vc{\pi}}_n - \vc{\pi}
\|$, which is the main result of this paper. To establish the bound,
we use the $\vc{v}$-norm, where $\vc{v}=(v(k,i))_{(k,i)\in\bbF}$ is
any nonnegative column vector. The $\vc{v}$-norm is defined as
follows: For any $1 \times |\bbF|$ vector
$\vc{x}=(x(k,i))_{(k,i)\in\bbF}$,
\[
\left\| \vc{x} \right\|_{\svc{v}} 
= \sup_{|\svc{g}| \le \svc{v}} 
\left| \sum_{(k,i)\in\bbF}x(k,i)g(k,i) \right|
= \sup_{\svc{0} \le \svc{g} \le \svc{v}}\sum_{(k,i)\in\bbF}|x(k,i)|g(k,i),
\]
where $|\vc{g}|$ is a column vector obtained by taking the
absolute value of each element of $\vc{g}$. By definition, $\| \cdot
\|_{\vc{e}} = \| \cdot \|$, i.e., the $\vc{e}$-norm is equivalent to
the total variation norm.

We need some further notations.  For $m \in \bbZ_+$ and
$(k,i)\in\bbF$, let $\vc{p}^m(k,i) = (p^m(k,i;l,j))_{(l,j)\in\bbF}$
and $\trunc{\vc{p}}_n^m(k,i) =(\trunc{p}_n^m(k,i;l,j))_{(l,j)\in\bbF}$
denote probability vectors such that $p^m(k,i;l,j)$ and
$\trunc{p}_n^m(k,i;l,j)$ represent the $(k,i;l,j)$th elements of
$\vc{P}^m$ and $(\trunc{\vc{P}}_n)^m$, respectively (when $m=1$, the
superscript ``1" may be omitted). Clearly, $p^m(k,i;l,j) =\PP(X_m = l,
J_m = j \mid X_0 = k, J_0 = i)$ for $(k,i) \times (l,j)\in\bbF^2$.

Let $\varpi(i) = \sum_{k=0}^{\infty}\pi(k,i) > 0$ for $i \in \bbD$. Note
here that if $\vc{P} \in \sfBM_d$, then
$\vc{\varpi}=(\varpi(i))_{i\in\bbD}$ is the stationary distribution of
$\vc{\varPsi}$ (and thus the Markov chain $\{J_{\nu}\}$; see
Proposition~\ref{prop-gamma(i,j)}). Note also that if $\vc{P} \in
\sfBM_d$, then $\trunc{\vc{P}}_n \prec_d \vc{P}$ and thus
$\trunc{\vc{\pi}}_n \prec_d \vc{\pi}$ (due to
Proposition~\ref{prop-4}~(c)), which implies that for all $n \in
\bbN$,
\begin{equation}
\sum_{k=0}^{\infty}\trunc{\pi}_n(k,i) 
= \sum_{k=0}^{\infty}\pi(k,i) = \varpi(i),\qquad i \in \bbD.
\label{eqn-varpi(i)}
\end{equation}
For any function $\varphi(\cdot,\cdot)$ on $\bbF$, let
$\varphi(k,\vc{\varpi}) = \sum_{i\in\bbD} \varpi(i)\varphi(k,i)$ for
$k \in \bbZ_+$.

In what follows, we estimate $\| \trunc{\vc{\pi}}_n - \vc{\pi}
\|$. By the triangle inequality, we have
\begin{eqnarray}
\left\| \trunc{\vc{\pi}}_n - \vc{\pi} \right\|
&\le& \left\| \vc{p}^m(0,\vc{\varpi}) - \vc{\pi} \right\|
+ \left\| \trunc{\vc{p}}_n^m(0,\vc{\varpi}) - \trunc{\vc{\pi}}_n \right\|
\nonumber
\\
&& {} \quad 
+ \left\| \trunc{\vc{p}}_n^m(0,\vc{\varpi}) - \vc{p}^m(0,\vc{\varpi}) \right\|.
\label{add-eqn-14}
\end{eqnarray}
The third term on the right hand side of (\ref{add-eqn-14}) is bounded
as in the following lemma, which is proved without $\vc{P} \in
\sfBM_d$.
\begin{lem}\label{appen-lem-1}
For all $m \in \bbN$, 
\begin{eqnarray}
\left\|  
\trunc{\vc{p}}_n^m(k,i) - \vc{p}^m(k,i)
\right\|
&\le& \sum_{h=0}^{m-1}
\sum_{(l,j)\in\bbF}\trunc{p}_n^h(k,i;l,j)\Delta_n(l,j),\quad
n \in \bbN,~(k,i) \in \bbF, \qquad~
\label{eqn-14a}
\end{eqnarray}
where 
\begin{equation}
\Delta_n(l,j) = \left\| \vc{p}(l,j) - \trunc{\vc{p}}_n(l,j)
\right\| = 2\sum_{l'>n,j'\in\bbD}p(l,j;l',j'),
\qquad (l,j)\in\bbF.
\label{defn-Delta(l,j)}
\end{equation}
\end{lem}

\begin{proof}
Clearly (\ref{eqn-14a}) holds for $m = 1$. Note here that for
$m,n\in\bbN$,
\[
(\trunc{\vc{P}}_n)^{m+1} - \vc{P}^{m+1}
= \trunc{\vc{P}}_n \cdot \left[ (\trunc{\vc{P}}_n)^m - \vc{P}^m \right]
+ (\trunc{\vc{P}}_n - \vc{P})\vc{P}^m.
\]
It then follows that for $m=2,3,\dots$, 
\begin{eqnarray}
\lefteqn{
\left\|  
\trunc{\vc{p}}_n^{m+1}(k,i) - \vc{p}^{m+1}(k,i)
\right\|
}
\quad && 
\nonumber
\\
&\le& \sum_{(l,j)\in\bbF} \trunc{p}_n(k,i;l,j) 
\left\|
\trunc{\vc{p}}_n^m(l,j) - \vc{p}^m(l,j)
\right\|
\nonumber
\\
&& {} \qquad
+ 
\sum_{(l,j)\in\bbF} 
| \trunc{p}_n(k,i;l,j) - p(k,i;l,j) |
\sum_{(l',j')\in\bbF} 
p^m(l,j;l',j')
\nonumber
\\
&=& \sum_{(l,j)\in\bbF} \trunc{p}_n(k,i;l,j) 
\left\| 
\trunc{\vc{p}}_n^m(l,j) - \vc{p}^m(l,j)
\right\|
+ \Delta_n(k,i),
\label{add-ineqn-01}
\end{eqnarray}
where the last equality is due to $\sum_{(l',j')\in\bbF}
p^m(l,j;l',j')=1$.  Thus if (\ref{eqn-14a}) holds for some $m \ge 2$,
then (\ref{add-ineqn-01}) yields
\begin{eqnarray*}
\lefteqn{
\left\|
\trunc{\vc{p}}_n^{m+1}(k,i) - \vc{p}^{m+1}(k,i) 
\right\|
}
&&
\nonumber
\\
&\le& \sum_{(l,j)\in\bbF} \trunc{p}_n(k,i;l,j) 
\left[
\sum_{h=0}^{m-1}
\sum_{(l',j')\in\bbF}\trunc{p}_n^h(l,j;l',j')\Delta_n(l',j')
\right]
 + \Delta_n(k,i)
\nonumber
\\
&=& \sum_{h=0}^{m-1}\sum_{(l',j')\in\bbF} 
\left(\sum_{(l,j)\in\bbF} \trunc{p}_n(k,i;l,j) \trunc{p}_n^h(l,j;l',j')
\right)\Delta_n(l',j')
 + \Delta_n(k,i)
\nonumber
\\
&=& 
 \sum_{h=0}^{m-1}
\sum_{(l',j')\in\bbF}\trunc{p}_n^{h+1}(k,i;l',j')
 \Delta_n(l',j')
+ 
\Delta_n(k,i)
=
 \sum_{h=0}^{m}
\sum_{(l,j)\in\bbF}\trunc{p}_n^h(k,i;l,j)
 \Delta_n(l,j).
\end{eqnarray*}
\end{proof}

The following lemma implies that the first two terms on the right hand
side of (\ref{add-eqn-14}) converge to zero as $m \to \infty$ without
the aperiodicity of $\vc{P}$.
\begin{lem}\label{lem-converge-P^m}
Let $\kappa$ denote the period of $\vc{P}$. If $\vc{P} \in \sfBM_d$
and $\vc{P}$ is irreducible, then the following hold:
\begin{enumerate}
\item There exist disjoint nonempty sets
  $\bbD_0,\bbD_1,\dots,\bbD_{\kappa-1}$ such that $\bbD =
  \cup_{h=0}^{\kappa-1}\bbD_{h}$ and
\[
\sum_{(l,j)\in\bbZ_+\times\bbD_{h+1}}p(k,i;l,j) = 1,
\quad (k,i) \in \bbZ_+\times\bbD_h,~~h \in \bbZ_+^{\leqslant \kappa-1},
\]
where $\bbD_{h'} = \bbD_h$ if $h' \equiv h$ ($\Mod ~\kappa$).
\item $\kappa \le d = |\bbD|$. Thus every irreducible monotone
  stochastic matrix (which is in $\sfBM_1$) is aperiodic.
\item If $\vc{P}$ is positive recurrent, then for $k \in \bbZ_+$,
\begin{eqnarray}
\lim_{m\to\infty}\vc{p}^m(k,\vc{\varpi}) = \vc{\pi},
\qquad
\lim_{m\to\infty}\trunc{\vc{p}}_n^m(k,\vc{\varpi}) = \trunc{\vc{\pi}}_n,
\quad n\in\bbN.
\label{lim-p^m}
\end{eqnarray}
\end{enumerate}
\end{lem}

\begin{proof}
We prove statement (a) by contradiction. Proposition~5.4.2 in
\cite{Meyn09} shows that there exist disjoint nonempty sets
$\bbF_0,\bbF_1,\dots,\bbF_{\kappa-1}$ such that $\bbF =
\cup_{h=0}^{\kappa-1}\bbF_{h}$ and
\begin{equation}
\sum_{(l,j)\in\bbF_{h+1}}p(k,i;l,j) = 1,
\quad (k,i) \in \bbF_h,~~h \in \bbZ_+^{\leqslant \kappa-1},
\label{eqn-aperiodic-P}
\end{equation}
where $\bbF_{h'} = \bbF_h$ if $h' \equiv h$ ($\Mod ~\kappa$). We
suppose that there exist some $(k_{\ast},i_{\ast}) \in \bbN \times
\bbD$ and $h_{\ast} \in \bbZ_+^{\leqslant \kappa-1}$ such that
$(0,i_{\ast}) \in \bbF_{h_{\ast}}$ and $(k_{\ast},i_{\ast}) \not\in
\bbF_{h_{\ast}}$. We now consider coupled Markov chains
$\{(X'_{\nu},J'_{\nu});\nu\in\bbZ_+\}$ and
$\{(X''_{\nu},J''_{\nu});\nu\in\bbZ_+\}$ with transition probability
matrix $\vc{P}$, which are pathwise ordered such that $X'_{\nu}\le
X''_{\nu}$ and $J'_{\nu} = J''_{\nu}$ for all $\nu \in \bbN$ if
$X_0'\le X_0''$ and $J_0' = J_0''$ (see
Lemma~\ref{lem1-discrete-ordering}).  We also fix $(X'_0,J'_0) =
(0,i_{\ast}) \in \bbF_{h_{\ast}}$ and $(X''_0,J'_0) =
(k_{\ast},i_{\ast}) \not\in \bbF_{h_{\ast}}$. It then follows from
(\ref{eqn-aperiodic-P}) that
\begin{equation}
\mbox{$(X'_{\nu},J'_{\nu}) \in \bbF_h$ implies 
$(X''_{\nu},J''_{\nu}) \not\in \bbF_h$ \quad for all $\nu \in \bbN$. }
\label{add-eqn-23}
\end{equation}
Further since $\vc{P}$ is irreducible, there exists some $\nu_{\ast}
\in \bbN$ such that $(X''_{\nu_{\ast}},J''_{\nu_{\ast}}) =
(0,i_{\ast})$ and thus $(X'_{\nu_{\ast}},J'_{\nu_{\ast}}) \in \bbN
\times \{i_{\ast}\}$ due to (\ref{add-eqn-23}).  This conclusion,
however, contradicts the pathwise ordering of
$\{(X'_{\nu},J'_{\nu})\}$ and $\{(X''_{\nu},J''_{\nu})\}$, i.e.,
$X'_{\nu} \le X''_{\nu}$ and $J'_{\nu} = J''_{\nu}$ for all
$\nu\in\bbN$. As a result, statement (a) holds, and statement (b) is
immediate from statement (a).

Next we prove statement (c). Fix $k \in \bbZ_+$ arbitrarily. Let
$q:\bbD \mapsto \bbZ_+^{\leqslant \kappa-1}$ denote a surjection
function such that $i \in \bbD_{q(i)}$. It then follows from Theorem 4
in Section I.6 of \cite{Chun67} that for $h \in \bbZ_+^{\leqslant
  \kappa-1}$,
\begin{equation}
\lim_{m' \to\infty}p^{m' \kappa + h}(k,i;l,j) 
= \dd{I}_{\{h \equiv q(j) - q(i) \,(\Mod \, \kappa)\}} \cdot \kappa \pi(l,j),
\qquad (l,j) \in \bbF,
\label{add-eqn-24}
\end{equation}
where $\dd{I}_{\{ \cdot \}}$ denotes a function that takes value one
if the statement in the braces is true and takes value zero otherwise.
From (\ref{add-eqn-24}), we have for $h \in \bbZ_+^{\leqslant
  \kappa-1}$ and $(l,j) \in \bbF$,
\begin{eqnarray}
\lim_{m' \to\infty}\sum_{i\in\bbD} \varpi(i) p^{m' \kappa + h}(k,i;l,j)
&=&
\lim_{m' \to\infty}\sum_{h'=0}^{\kappa-1} \sum_{i\in\bbD_{h'}} \varpi(i)
p^{m' \kappa + h}(k,i;l,j)
\nonumber
\\
&=& \kappa
\sum_{h'=0}^{\kappa-1} \sum_{i\in\bbD_{h'}} \varpi(i)
\dd{I}_{\{h \equiv q(j) - q(i) \,(\Mod \, \kappa)\}} \cdot \pi(l,j)
\nonumber
\\
&=& \kappa
\sum_{h'=0}^{\kappa-1} \sum_{i\in\bbD_{h'}} \varpi(i)
\dd{I}_{\{h \equiv q(j) - h' \,(\Mod \, \kappa)\}} \cdot \pi(l,j),
\qquad
\label{add-eqn-25}
\end{eqnarray}
where the last equality is due to $q(i) = h'$ for $i\in\bbD_{h'}$.
Note here that $\sum_{i\in\bbD_{h'}} \varpi(i) =
\sum_{(k,i)\in\bbF_{h'}} \pi(k,i)= 1/\kappa$ for any $h' \in
\bbZ_+^{\leqslant \kappa-1}$ (see Theorem 1 in Section I.7 of
\cite{Chun67}). Note also that for any $h \in \bbZ_+^{\leqslant
  \kappa-1}$ and $j\in\bbD$ there exists the unique $h' \in
\bbZ_+^{\leqslant \kappa-1}$ such that $h \equiv q(j) - h'$ $(\Mod
~\kappa)$.  From (\ref{add-eqn-25}), we then obtain for $h \in
\bbZ_+^{\leqslant \kappa-1}$,
\[
\lim_{m' \to\infty}\sum_{i\in\bbD} \varpi(i) p^{m' \kappa + h}(k,i;l,j)
= \pi(l,j),\qquad (l,j) \in \bbF,
\]
which leads to the first limit in (\ref{lim-p^m}). Further since
$\trunc{\vc{P}}_n \prec_d \vc{P} \in \sfBM_d$, it follows from
Proposition~\ref{prop-4}~(c) that $\trunc{\vc{P}}_n$ has the unique
positive recurrent class. As a result, we can prove the second limit
in (\ref{lim-p^m}) in the same way as the proof of the first one.
\end{proof}

To estimate the first two terms on the right hand side of
(\ref{add-eqn-14}), we assume the geometric drift condition for
geometric ergodicity:
\begin{assumpt}\label{assumpt-geo}
There exists a column vector $\vc{v}=(v(k,i))_{(k,i)\in\bbF} \in
\sfBI_d$ such that $\vc{v} \ge \vc{e}$ and for some $\gamma \in (0,1)$
and $b \in (0, \infty)$,
\begin{equation}
\vc{P}\vc{v} \le \gamma \vc{v} + b\dd{1}_0,
\label{ineqn-Pv}
\end{equation}
where $\dd{1}_K=(1_K(k,i))_{(k,i)\in\bbF}$ ($K \in \bbZ_+$) denotes a
column vector such that $1_K(k,i)=1$ for $(k,i) \in \bbF^{\leqslant
  K}$ and $1_K(k,i) = 0$ for $(k,i) \in \bbF \setminus \bbF^{\leqslant
  K}$.
\end{assumpt}

\begin{rem}
Suppose that $\vc{P}$ is irreducible. Since the state space $\bbF$ is
countable, every subset of $\bbF$ includes a {\it small set} and thus
{\it petite set} (see Theorem 5.2.2 and Proposition~5.5.3 in
\cite{Meyn09}). Therefore if the irreducible $\vc{P}$ is aperiodic and
Assumption~\ref{assumpt-geo} holds, then there exist $r \in
(1,\infty)$ and $C \in (0,\infty)$ such that $\sum_{m=1}^{\infty} r^m
\left\| \vc{p}^m(k,i) - \vc{\pi} \right\|_{\svc{v}} \le C v(k,i)$ for
all $(k,i) \in \bbF$, which shows that $\vc{P}$ is
$\vc{v}$-geometrically ergodic (see Theorem 15.0.1 in \cite{Meyn09}).
\end{rem}

The following lemma is an extension of Theorem~2.2 in \cite{Lund96} to
discrete-time BMMCs. 
\begin{lem}\label{lem-two-bounds}
Suppose that $\vc{P} \in \sfBM_d$ and $\vc{P}$ is irreducible. If
Assumption~\ref{assumpt-geo} holds, then for all $k \in \bbZ_+$ and
$m\in\bbN$,
\begin{eqnarray}
\left\|  \vc{p}^m(k,\vc{\varpi}) - \vc{\pi}  \right\|_{\svc{v}}
&\le& 2\gamma^m \left[v(k,\vc{\varpi})(1 - 1_0(k,\vc{\varpi})) 
+ b/(1 - \gamma)
\right],
\label{eqn-13}
\\
\left\|  
\trunc{\vc{p}}_n^m(k,\vc{\varpi}) - \trunc{\vc{\pi}}_n  
\right\|_{\svc{v}}
&\le& 2\gamma^m \left[v(k,\vc{\varpi})(1 - 1_0(k,\vc{\varpi})) 
+ b/(1 - \gamma)
\right],~~\forall n \in \bbN.\qquad
\label{add-eqn-13}
\end{eqnarray}
\end{lem}

\begin{proof}
We first prove (\ref{eqn-13}). To do this, we consider three copies
$\{(X_{\nu}^{(h)},J_{\nu}^{(h)});\nu\in\bbZ_+\}$ ($h=0,1,2$) of the
BMMC $\{(X_{\nu},J_{\nu});\nu\in\bbZ_+\}$, which are defined on a
common probability space in such a way that
\[
(X_0^{(0)},J_0^{(0)}) =(0,J),  \quad 
(X_0^{(1)},J_0^{(1)}) = (k,J), \quad 
(X_0^{(2)},J_0^{(2)}) = (X,J),
\]
where $k \in \bbZ_+$ and $(X,J)$ denotes a random vector distributed
with $\PP(X = l, S=j) = \pi(l,j)$ for $(l,j) \in \bbF$. According to
the pathwise ordered property of BMMCs (see
Lemma~\ref{lem1-discrete-ordering}), we assume without loss of
generality that
\begin{equation}
X_{\nu}^{(0)} \le X_{\nu}^{(1)}, \quad 
X_{\nu}^{(0)} \le X_{\nu}^{(2)}, \quad
J_{\nu}^{(0)} = J_{\nu}^{(1)} = J_{\nu}^{(2)},
\quad \forall\nu \in \bbZ_+.
\label{eqn-01}
\end{equation}

For simplicity, let
\begin{align*}
\EE_{(k,i)}[\,\, \cdot \,\,] 
&= \EE[~ \cdot \mid X_0=k, J_0=i], 
& (k,i) &\in \bbF,
\\
\EE_{(k,i);(0,j)}[\,\, \cdot \,\,] 
&= \EE[~  \cdot \mid (X_0^{(h)}, 
J_0^{(h)}) = (k,i), (X_0^{(0)},J_0^{(0)}) = (0,j)],
& (k,i) &\in \bbF,~j\in\bbD,
\end{align*}
where $h=1,2$. Further let $\vc{g} =(g(l,j))_{(l,j) \in \bbF}$ denote
a column vector satisfying $|\vc{g}| \le \vc{v}$, i.e., $|g(l,j)| \le
v(l,j)$ for $(l,j)\in\bbF$. It then follows that for $m=1,2,\dots$,
\begin{eqnarray*}
\vc{p}^m(k,\vc{\varpi})\vc{g}
&=& 
\sum_{i \in \bbD}\varpi(i)\sum_{(l,j) \in \bbF}p^m(k,i;l,j)g(l,j)
= \EE\!\left[\EE_{(k,J)}[g(X_m,J_m)] \right],
\label{eqn-04}
\\
\vc{\pi}\vc{g} 
&=& \vc{\pi}\vc{P}^m\vc{g}
= 
\sum_{(k,i) \in \bbF}\pi(k,i)
\sum_{(l,j) \in \bbF}p^m(k,i;l,j)g(l,j)
= \EE[\EE_{(X,J)}[g(X_m,J_m)]]. \quad
\label{eqn-05}
\end{eqnarray*}
Thus by the triangle inequality, we obtain
\begin{eqnarray}
\lefteqn{
|\vc{p}^m(k,\vc{\varpi})\vc{g} - \vc{\pi}\vc{g}|
}
\quad&&
\nonumber
\\
&=& 
\left|\EE\!\left[\EE_{(k,J)}[g(X_m,J_m)] \right]  
- \EE\!\left[\EE_{(X,J)}[g(X_m,J_m)]\right] \right|
\nonumber
\\
&\le& \left|\EE\!\left[\EE_{(k,J);(0,J)}[g(X_m^{(1)},J_m^{(1)})] \right]  
- \EE\!\left[\EE_{(k,J);(0,J)}[g(X_m^{(0)},J_m^{(0)})]\right] \right|
\nonumber
\\
&& {} 
+
\left|\EE\!\left[\EE_{(X,J);(0,J)}[g(X_m^{(2)},J_m^{(2)})] \right]  
- \EE\!\left[\EE_{(X,J);(0,J)}[g(X_m^{(0)},J_m^{(0)})]\right] \right|.
\label{eqn-05'}
\end{eqnarray}

Let $T_h = \inf\{m\in\bbZ_+;
X_{\nu}^{(h)}=X_{\nu}^{(0)},~\forall \nu\ge m\}$ for $h=1,2$.  We
then have
\begin{eqnarray}
g(X_{\nu}^{(1)},J_{\nu}^{(1)}) &=& g(X_{\nu}^{(0)},J_{\nu}^{(0)}), 
\qquad \nu \ge T_1,
\label{eqn-02}
\\
g(X_{\nu}^{(2)},J_{\nu}^{(2)}) &=& g(X_{\nu}^{(0)},J_{\nu}^{(0)}), 
\qquad \nu \ge T_2.
\label{eqn-03}
\end{eqnarray}
Applying (\ref{eqn-02}) and (\ref{eqn-03}) to (\ref{eqn-05'}) and
using $|\vc{g}| \le \vc{v}$ (but not $\vc{P} \in \sfBM_d$) yield
\begin{eqnarray}
\lefteqn{
|\vc{p}^m(k,\vc{\varpi})\vc{g} - \vc{\pi}\vc{g}|
}
\quad&&
\nonumber
\\
&\le& \EE\!\left[\EE_{(k,J);(0,J)}[|g(X_m^{(1)},J_m^{(1)}) - g(X_m^{(0)},J_m^{(0)})| 
\cdot \dd{I}_{\{T_1 > m\}} ]\right] 
\nonumber
\\
&& {} 
+
\EE\!\left[\EE_{(X,J);(0,J)}[|g(X_m^{(2)},J_m^{(2)}) - g(X_m^{(0)},J_m^{(0)})| 
\cdot \dd{I}_{\{T_2 > m\}} ]\right]
\nonumber
\\
&\le& 
\EE\!\left[\EE_{(k,J);(0,J)}[v(X_m^{(1)},J_m^{(1)}) \cdot \dd{I}_{\{T_1 > m\}}] \right]
\nonumber
\\
&& {} 
+ \EE\!\left[\EE_{(k,J);(0,J)}[v(X_m^{(0)},J_m^{(0)}) \cdot \dd{I}_{\{T_1 > m\}}]\right]
\nonumber
\\
&& {} + 
\EE\!\left[\EE_{(X,J);(0,J)}[v(X_m^{(2)},J_m^{(2)}) \cdot \dd{I}_{\{T_2 > m\}}] \right]
\nonumber
\\
&& {} 
+ \EE\!\left[\EE_{(X,J);(0,J)}[v(X_m^{(0)},J_m^{(0)}) \cdot \dd{I}_{\{T_2 > m\}}]\right].
\label{eqn-06a}
\end{eqnarray}
Combining (\ref{eqn-06a}) with (\ref{eqn-01}) and $\vc{v} \in
\sfBI_d$, we obtain for all $|\vc{g}| \le \vc{v}$,
\begin{eqnarray}
|\vc{p}^m(k,\vc{\varpi})\vc{g} - \vc{\pi}\vc{g}|
&\le& 2\EE\!\left[\EE_{(k,J);(0,J)}[v(X_m^{(1)},J_m^{(1)}) \cdot \dd{I}_{\{T_1 > m\}}] \right]
\nonumber
\\
&& {} 
+ 2 \EE\!\left[\EE_{(X,J);(0,J)}[v(X_m^{(2)},J_m^{(2)}) \cdot \dd{I}_{\{T_2 > m\}}] \right].
\label{eqn-06b}
\end{eqnarray}
Further it follows from (\ref{eqn-01}) that $X_m^{(h)} = 0$ ($h=1,2$)
implies $X_{\nu}^{(h)} = X_{\nu}^{(0)}$ for all $\nu \ge m$, which
leads to $T_h \le \inf\{\nu\in\bbZ_+;X_{\nu}^{(h)}=0\}$
($h=1,2$). Thus we have
\begin{eqnarray}
\EE\!\left[\EE_{(k,J);(0,J)}[v(X_m^{(1)},J_m^{(1)}) \cdot \dd{I}_{\{T_1 > m\}}] \right]
&\le& \EE\!\left[\EE_{(k,J)}[v(X_m,J_m) \cdot \dd{I}_{\{\tau_0 > m\}}] \right],
\label{eqn-07}
\\
\EE\!\left[\EE_{(X,J);(0,J)}[v(X_m^{(2)},J_m^{(2)}) \cdot \dd{I}_{\{T_2 > m\}}] \right]
&\le& \EE\!\left[\EE_{(X,J)}[v(X_m,J_m) \cdot \dd{I}_{\{\tau_0 > m\}}] \right],
\quad
\label{eqn-08}
\end{eqnarray}
where $\tau_0 = \inf\{\nu\in\bbZ_+;X_{\nu}=0\}$. Substituting
(\ref{eqn-07}) and (\ref{eqn-08}) into (\ref{eqn-06b}) yields
\begin{eqnarray}
\left\| \vc{p}^m(k,\vc{\varpi}) - \vc{\pi}  \right\|_{\svc{v}}\,
&\le& 2\EE\!\left[\EE_{(k,J)}[v(X_m,J_m) \cdot \dd{I}_{\{\tau_0 > m\}}] \right]
\nonumber
\\
&& {} + 2\EE\!\left[\EE_{(X,J)}[v(X_m,J_m) \cdot \dd{I}_{\{\tau_0 > m\}}] \right].
\label{eqn-09}
\end{eqnarray}

Let $M_m=\gamma^{-m}v(X_m,J_m)\dd{I}_{\{\tau_0 > m\}}$ for $m \in
\bbZ_+$. If $\tau_0 \le m$, $M_{m+1} = M_m = 0$. On the other hand,
suppose that $\tau_0 > m$ and thus $(X_m,J_m) = (k,i) \in \bbN \times
\bbD$ (due to $\{\tau_0 > m\} \subseteq \{X_m \in \bbN\}$). We then
have for $(k,i) \in \bbN \times \bbD$,
\begin{eqnarray*}
\EE[M_{m+1} \mid (X_m,J_m) = (k,i), \tau_0 > m]
&=& \sum_{(l,j)\in\bbN\times\bbD}p(k,i;l,j)\gamma^{-m-1}v(l,j)
\nonumber
\\
&\le& \sum_{(l,j) \in \bbF}p(k,i;l,j)\gamma^{-m-1}v(l,j)
\le \gamma^{-m}v(k,i),
\end{eqnarray*}
where the last inequality follows from (\ref{ineqn-Pv}). Thus
$\{M_m\}$ is a supermartingale.

Let $\{\theta_{\nu};\nu\in \bbZ_+\}$ denote a sequence of stopping
times for $\{M_m;m\in\bbZ_+\}$ such that $0 \le \theta_1 \le \theta_2
\le \cdots$ and $\lim_{\nu\to\infty}\theta_{\nu}=\infty$. Note that
for any $m' \in \bbZ_+$, $\min(m',\theta_{\nu})$ is a stopping time
for $\{M_m;m\in\bbZ_+\}$. It then follows from Doob's optional
sampling theorem that for $(k,i) \in \bbF$,
$\EE_{(k,i)}[M_{\min(m,\theta_{\nu})}] \le \EE_{(k,i)}[M_0]$, i.e.,
\begin{eqnarray*}
\EE_{(k,i)}[\gamma^{-\min(m,\theta_{\nu})}
v(X_{\min(m,\theta_{\nu})},J_{\min(m,\theta_{\nu})})
\dd{I}_{\{\tau_0 > \min(m,\theta_{\nu}) \}}]
 \le v(k,i)(1 - 1_0(k,i)).
\end{eqnarray*}
Thus letting $\nu\to\infty$ and using Fatou's lemma, we have
\begin{equation}
\EE_{(k,i)}[v(X_m,J_m) \dd{I}_{\{\tau_0 > m \}}]
\le \gamma^m  v(k,i)(1 - 1_{0}(k,i)),
\label{eqn-10}
\end{equation}
which leads to
\begin{eqnarray}
\EE\!\left[\EE_{(k,J)}[v(X_m,J_m)\dd{I}_{\{\tau_0 > m \}}] \right]
&=& \sum_{i\in\bbD}\varpi(i)\EE_{(k,i)}[v(X_m,J_m) \dd{I}_{\{\tau_0 > m \}}]
\nonumber
\\
&\le& \gamma^m  v(k,\vc{\varpi})(1 - 1_{0}(k,\vc{\varpi})),
\label{eqn-11}
\end{eqnarray}
where we use $1_0(k,i) = 1_0(k,\vc{\varpi})$ for all $i \in
\bbD$. Note here that pre-multiplying both sides of (\ref{ineqn-Pv})
by $\vc{\pi}$ yields $\vc{\pi}\vc{v} \le b /(1 - \gamma)$, 
from which and (\ref{eqn-10}) we obtain
\begin{equation}
\EE\!\left[\EE_{(X,J)}[v(X_m,J_m) \cdot \dd{I}_{\{\tau_0 > m\}}] \right]
\le \gamma^m \sum_{(k,i)\in\bbF}\pi(k,i)v(k,i)
\le \gamma^m {b \over 1 - \gamma}.
\label{eqn-12}
\end{equation}
Substituting (\ref{eqn-11}) and (\ref{eqn-12}) into (\ref{eqn-09})
yields (\ref{eqn-13}). 

Next we consider (\ref{add-eqn-13}). Since $\vc{P} \in \sfBM_d$, we
have $\trunc{\vc{P}}_n \in \sfBM_d$ and $\trunc{\vc{P}}_n \prec_d
\vc{P}$. Thus since $\vc{P}$ is irreducible and positive recurrent,
Proposition~\ref{prop-4}~(c) implies that $\trunc{\vc{P}}_n$ has the
unique positive recurrent class, which includes the states
$\{(0,i);i\in\bbD\}$. Further it follows from $\vc{v} \in \sfBI_d$,
(\ref{ineqn-Pv}) and Remark~\ref{prop-3} that
\begin{equation}
\trunc{\vc{P}}_n\vc{v} \le \vc{P}\vc{v} \le \gamma\vc{v} + b\dd{1}_0.
\label{ineqn-(n)P_n*v}
\end{equation}
Therefore we can prove (\ref{add-eqn-13}) in the same way as the proof
of (\ref{eqn-13}).
\end{proof}

Combining (\ref{add-eqn-14}) with Lemmas~\ref{appen-lem-1} and
\ref{lem-two-bounds}, we obtain the following theorem.
\begin{thm}\label{thm-main-geo}
Suppose that $\vc{P} \in \sfBM_d$ and $\vc{P}$ is irreducible. If
Assumption~\ref{assumpt-geo} holds, then for all $n \in \bbN$,
\begin{eqnarray}
\left\| \trunc{\vc{\pi}}_n - \vc{\pi} \right\|
&\le& 4\gamma^m {b \over 1-\gamma} 
+ 2m \sum_{i\in\bbD}\trunc{\pi}_n(n,i),\quad \forall m \in \bbN,
\label{geo-bound-1}
\\
\left\| \trunc{\vc{\pi}}_n - \vc{\pi} \right\|
&\le& {b \over 1-\gamma}
\left(
4\gamma^m  + 2m \sum_{i\in\bbD}{1 \over v(n,i)}
\right),\quad \forall m \in \bbN.
\label{geo-bound-2}
\end{eqnarray}
\end{thm}

\begin{rem}
If $d=1$, Theorem~\ref{thm-main-geo} is reduced to Theorem~4.2 in
\cite{Twee98}.
\end{rem}

\begin{proof}[Proof of Theorem~\ref{thm-main-geo}]
From (\ref{add-eqn-14}) and Lemma~\ref{lem-two-bounds}, we have
\begin{equation}
\left\| \trunc{\vc{\pi}}_n - \vc{\pi} \right\|
\le 4\gamma^m {b \over 1 - \gamma}
+ \left\| \trunc{\vc{p}}_n^m(0,\vc{\varpi}) - \vc{p}^m(0,\vc{\varpi}) \right\|.
\label{eqn-15}
\end{equation}
From Lemma~\ref{appen-lem-1} (which does not require $\vc{P} \in
\sfBM_d$), we obtain for $m \in \bbN$,
\begin{eqnarray}
\left\|
\trunc{\vc{p}}_n^m(0,\vc{\varpi}) - \vc{p}^m(0,\vc{\varpi})
\right\|
&\le& \sum_{i\in\bbD}\varpi(i) \left\|
\trunc{\vc{p}}_n^m(0,i) - \vc{p}^m(0,i)
\right\|
\nonumber
\\
&\le& \sum_{h=0}^{m-1} 
\sum_{(l,j)\in\bbF} 
\left( \sum_{i\in\bbD} \varpi(i)\trunc{p}_n^h(0,i;l,j) \right) \Delta_n(l,j). 
\qquad
\label{eqn-14}
\end{eqnarray}
It follows from (\ref{eqn-varpi(i)}) and $\trunc{\vc{P}}_n \in
\sfBM_d$ that $(\vc{\varpi},0,0,\dots) \prec_d \trunc{\vc{\pi}}_n$ and
$(\trunc{\vc{P}}_n)^h \in \sfBM_d$ for $h \in \bbN$. Thus
Proposition~\ref{prop-2} yields
\begin{equation}
(\vc{\varpi},0,0,\dots)(\trunc{\vc{P}}_n)^h 
\prec_d \trunc{\vc{\pi}}_n(\trunc{\vc{P}}_n)^h 
= \trunc{\vc{\pi}}_n.
\label{add-eqn-01}
\end{equation}
In addition, $\vc{P} \in \sfBM_d$ and (\ref{defn-Delta(l,j)}) imply
that a column vector
$\vec{\vc{\delta}}_n:=(\Delta_n(l,j))_{(l,j)\in\bbF}$ with block size
$d$ is block-increasing, i.e., $\vec{\vc{\delta}}_n \in
\sfBI_d$. Combining this and (\ref{add-eqn-01}) with
Remark~\ref{prop-3}, we have
\[
(\vc{\varpi},0,0,\dots)(\trunc{\vc{P}}_n)^h \vec{\vc{\delta}}_n
\le \trunc{\vc{\pi}}_n\vec{\vc{\delta}}_n.
\]
Applying (\ref{defn-Delta(l,j)}) to the right hand side of the above
inequality, we obtain
\begin{eqnarray}
\lefteqn{
\sum_{(l,j)\in\bbF}
\left( \sum_{i\in\bbD} \varpi(i)\trunc{p}_n^h(0,i;l,j) \right)\Delta_n(l,j)
}
\qquad &&
\nonumber
\\
&\le& 2
\sum_{(l,j)\in\bbF}\trunc{\pi}_n(l,j)
\sum_{l'>n,j'\in\bbD} p(l,j;l',j')
\nonumber
\\
&\le& 2
\sum_{(l,j)\in\bbF}\trunc{\pi}_n(l,j)
\sum_{j'\in\bbD}\trunc{p}_n(l,j;n,j')
=2
\sum_{j'\in\bbD}\trunc{\pi}_n(n,j'),
\label{eqn-16}
\end{eqnarray}
where the second inequality follows from
(\ref{def-trunc{p}_n(k,i;l,j)}) and the last equality follows from
$\trunc{\vc{\pi}}_n \cdot \trunc{\vc{P}}_n=\trunc{\vc{\pi}}_n$.
Substituting (\ref{eqn-16}) into (\ref{eqn-14}) yields
\[
\left\|
\trunc{\vc{p}}_n^m(0,\vc{\varpi}) - \vc{p}^m(0,\vc{\varpi})
\right\|
\le 2m\sum_{j'\in\bbD}\trunc{\pi}_n(n,j'),
\]
from which and (\ref{eqn-15}) we have (\ref{geo-bound-1}).

Next we prove (\ref{geo-bound-2}). Pre-multiplying both sides of 
(\ref{ineqn-(n)P_n*v}) by $\trunc{\vc{\pi}}_n$ and using
$\trunc{\vc{\pi}}_n \cdot \trunc{\vc{P}}_n = \trunc{\vc{\pi}}_n$, we
obtain $\trunc{\vc{\pi}}_n\vc{v} \le b/(1-\gamma)$, 
which leads to
\[
\trunc{\pi}_n(n,i)
\le
{b \over 1-\gamma}{1 \over v(n,i)},
\qquad i \in \bbD.
\]
Substituting this inequality into (\ref{geo-bound-1}) yields
(\ref{geo-bound-2}).
\end{proof}

\section{Extensions of main result}\label{sec-extension}

In this section, we do not necessarily assume that $\vc{P}$ (i.e.,
Markov chain $\{(X_{\nu},J_{\nu});\nu\in\bbZ_+\}$) is block-monotone,
but assume that $\vc{P}$ is block-wise dominated by an irreducible and
positive recurrent stochastic matrix in $\sfBM_d$, which is denoted by
$\widetilde{\vc{P}}=(\widetilde{p}(k,i;l,j))_{(k,i),(l,j)\in\bbF}$.
Let $\widetilde{\vc{\pi}}=(\widetilde{\pi}(k,i))_{(k,i)\in\bbF}$
denote the stationary probability vector of $\widetilde{\vc{P}}$. It
follows from $\vc{P} \prec_d \widetilde{\vc{P}} \in \sfBM_d$ and
Proposition~\ref{prop-4}~(c) that
$\vc{\pi}\prec_d\widetilde{\vc{\pi}}$ and thus
\begin{equation}
\sum_{k=0}^{\infty}\widetilde{\pi}(k,i)= \sum_{k=0}^{\infty}\pi(k,i)
= \varpi(i),
\qquad i \in \bbD.
\label{eqn-tilde{J}=J}
\end{equation}

Let $\{(\widetilde{X}_{\nu},\widetilde{J}_{\nu});\nu\in\bbZ_+\}$
denote a BMMC with state space $\bbF$ and transition probability
matrix $\widetilde{\vc{P}}$. Since $\vc{P} \prec_d \widetilde{\vc{P}}
\in \sfBM_d$, we can assume (without loss of generality) that the
pathwise ordering of $ \{(\widetilde{X}_{\nu},\widetilde{J}_{\nu})\}$
and $\{(X_{\nu},J_{\nu})\}$ holds, i.e., if $X_0 \le \widetilde{X}_0$
and $J_0=\widetilde{J}_0$, then $X_{\nu} \le \widetilde{X}_{\nu}$ and
$J_{\nu}= \widetilde{J}_{\nu}$ for all $n \in \bbN$ (see
Lemma~\ref{lem2-discrete-ordering}).

The following result is an extension of Theorem~5.1 in \cite{Twee98}.
\begin{thm}\label{thm-extended-geo}
Suppose that (i) $\widetilde{\vc{P}} \in \sfBM_d$ and
$\widetilde{\vc{P}}$ is irreducible; (ii) $\vc{P}
\prec_d \widetilde{\vc{P}}$; and (iii) there exists a column vector
$\vc{v}=(v(k,i))_{(k,i)\in\bbF} \in \sfBI_d$ such that $\vc{v} \ge
\vc{e}$ and
\begin{equation}
\widetilde{\vc{P}}\vc{v} \le \gamma \vc{v} + b\dd{1}_0,
\label{ineqn-hat{P}v}
\end{equation}
for some $\gamma \in (0,1)$ and $b \in (0, \infty)$. Under these
conditions, (\ref{geo-bound-2}) holds for all $n \in \bbN$.
\end{thm}

\begin{proof}
We first prove the two bounds (\ref{eqn-13}) and (\ref{add-eqn-13}).
Let $(X,J)$ and $(\widetilde{X},\widetilde{J})$ denote two random
vectors on a probability space $(\Omega,\calF,\PP)$ such that
$\PP(\widetilde{X}=k, \widetilde{J}=i) = \widetilde{\pi}(k,i)$ for
$(k,i) \in \bbF$. Note here that since $\vc{\pi} \prec_d
\widetilde{\vc{\pi}}$, $\sum_{l=k}^{\infty}\pi(l,i)/\varpi(i) \le
\sum_{l=k}^{\infty}\widetilde{\pi}(l,i)/\varpi(i)$ for
$(k,i)\in\bbF$. According to this inequality and
(\ref{eqn-tilde{J}=J}), we can assume that $X \le \widetilde{X}$ and
$J = \widetilde{J}$ (see Theorem 1.2.4 in \cite{Mull02}).  We then
introduce the copies
$\{(\widetilde{X}_{\nu}^{(h)},\widetilde{J}_{\nu}^{(h)})\}$ and
$\{(X_{\nu}^{(h)},J_{\nu}^{(h)})\}$ ($h=0,1,2$) of the Markov chains
$\{(\widetilde{X}_{\nu},\widetilde{J}_{\nu})\}$ and
$\{(X_{\nu},J_{\nu})\}$, respectively, on the common probability space
$(\Omega,\calF,\PP)$, where
\begin{eqnarray*}
&&
(\widetilde{X}_0^{(0)},\widetilde{J}_0^{(0)}) =(0,\widetilde{J}),  \quad 
(\widetilde{X}_0^{(1)},\widetilde{J}_0^{(1)}) = (k,\widetilde{J}), \quad 
(\widetilde{X}_0^{(2)},\widetilde{J}_0^{(2)}) = (\widetilde{X},\widetilde{J}),
\\
&&
(X_0^{(0)},J_0^{(0)}) =(0,J),  \quad 
(X_0^{(1)},J_0^{(1)}) = (k,J), \quad 
(X_0^{(2)},J_0^{(2)}) = (X,J).
\end{eqnarray*}
From the pathwise
ordering of $\{(\widetilde{X}_{\nu},\widetilde{J}_{\nu})\}$ and
$\{(X_{\nu},J_{\nu})\}$, we have for $h=0,1,2$,
\begin{equation}
X_{\nu}^{(h)} \le \widetilde{X}_{\nu}^{(h)},\quad
J_{\nu}^{(h)} = \widetilde{J}_{\nu}^{(h)},\quad \forall \nu\in\bbZ_+.
\label{ordering-01}
\end{equation}
In addition, by the pathwise ordered property of $\widetilde{\vc{P}} \in
\sfBM_d$ (see Lemma~\ref{lem1-discrete-ordering}), we assume that
\begin{equation}
\widetilde{X}_{\nu}^{(0)} \le \widetilde{X}_{\nu}^{(1)}, \quad 
\widetilde{X}_{\nu}^{(0)} \le \widetilde{X}_{\nu}^{(2)}, \quad
\widetilde{J}_{\nu}^{(0)} = \widetilde{J}_{\nu}^{(1)} = \widetilde{J}_{\nu}^{(2)},
\quad \forall \nu \in \bbZ_+.
\label{ordering-02}
\end{equation}

Let $\vc{g} =(g(l,j))_{(l,j) \in \bbF}$ denote a column vector
satisfying $|\vc{g}| \le \vc{v}$. It then follows that (\ref{eqn-06a})
holds under the assumptions of Theorem~\ref{thm-extended-geo} because
(\ref{eqn-06a}) does not require that $\{(X_{\nu},J_{\nu})\}$ is block
monotone. Further applying (\ref{ordering-01}), (\ref{ordering-02})
and $\vc{v} \in \sfBI_d$ to (\ref{eqn-06a}), we obtain for all
$|\vc{g}| \le \vc{v}$,
\begin{eqnarray}
|\vc{p}^m(k,\vc{\varpi})\vc{g} - \vc{\pi}\vc{g}|
&\le& 2\EE\!\left[
\EE_{(k,\widetilde{J});(0,\widetilde{J})}
[v(\widetilde{X}_m^{(1)},\widetilde{J}_m^{(1)}) \cdot \dd{I}_{\{T_1 > m\}}] 
\right]
\nonumber
\\
&& {} 
+ 2 \EE\!\left[
\EE_{(\widetilde{X},\widetilde{J});(0,\widetilde{J})}
[v(\widetilde{X}_m^{(2)},\widetilde{J}_m^{(2)}) \cdot \dd{I}_{\{T_2 > m\}}] 
\right],
\label{eqn-17}
\end{eqnarray}
where $T_h = \inf\{m\in\bbZ_+;
X_{\nu}^{(h)}=X_{\nu}^{(0)}~(\forall \nu\ge m)\}$ for $h=1,2$.

It follows from (\ref{ordering-01}) and (\ref{ordering-02}) that for
each $h \in \{1,2\}$, $\widetilde{X}_m^{(h)} = 0$ implies $X_m^{(h)} =
X_m^{(0)} = 0$ and thus $X_{\nu}^{(h)} = X_{\nu}^{(0)}$ for all $\nu
\ge m$, which leads to $T_h \le \inf\{\nu\in\bbZ_+;
\widetilde{X}_{\nu}^{(h)}=0\}$. Therefore from (\ref{eqn-17}), we can
obtain the following inequality (see the derivation of (\ref{eqn-09})
from (\ref{eqn-06b})):
\begin{eqnarray}
\left\|\vc{p}^m(k,\vc{\varpi}) - \vc{\pi}\right\|_{\svc{v}}
&\le& 2\EE\!\left[\EE_{(k,\widetilde{J})}[v(\widetilde{X}_m,\widetilde{J}_m) 
\cdot \dd{I}_{\{\widetilde{\tau}_0 > m\}}] 
\right]
\nonumber
\\
&& {}
+   2\EE\!\left[\EE_{(\widetilde{X},\widetilde{J})}[v(\widetilde{X}_m,\widetilde{J}_m) 
\cdot \dd{I}_{\{\widetilde{\tau}_0 > m\}}] \right],
\label{eqn-21}
\end{eqnarray}
where $\widetilde{\tau}_0 = \inf\{\nu\in\bbZ_+;\widetilde{X}_{\nu} =
0\}$. Further, following the discussion after (\ref{eqn-09}), we can
show that for all $k \in \bbZ_+$ and $m \in \bbN$,
\begin{eqnarray*}
\left\| \vc{p}^m(k,\vc{\varpi}) - \vc{\pi} \right\|_{\svc{v}}
&\le& 2\gamma^m \left[v(k,\vc{\varpi})(1 - 1_0(k,\vc{\varpi})) 
+ b/(1 - \gamma)
\right],
\\
\left\|
\trunc{\vc{p}}_n^m(k,\vc{\varpi}) - \trunc{\vc{\pi}}_n  
\right\|_{\svc{v}}
&\le& 2\gamma^m \left[v(k,\vc{\varpi})(1 - 1_0(k,\vc{\varpi})) 
+ b/(1 - \gamma)
\right],\quad \forall n \in \bbN. 
\end{eqnarray*}
Consequently, we obtain the two bounds (\ref{eqn-13}) and
(\ref{add-eqn-13}).

It remains to prove that
\begin{equation}
\left\| \trunc{\vc{p}}_n^m(0,\vc{\varpi}) - \vc{p}^m(0,\vc{\varpi}) \right\|
\le {2mb \over 1 - \gamma}\sum_{i\in\bbD}{1 \over v(n,i)}.
\label{add-eqn-19}
\end{equation}
Let $\widetilde{\Delta}_n(l,j) =
2\sum_{l'>n,j'\in\bbD}\widetilde{p}(l,j;l',j')$ for
$(l,j)\in\bbF$. Since $\vc{P} \prec_d \widetilde{\vc{P}}$, we have
$\Delta_n(l,j) \le \widetilde{\Delta}_n(l,j)$ for $(l,j)\in\bbF$. Note
here that (\ref{eqn-14}) still holds and thus
\begin{eqnarray}
\left\|
\trunc{\vc{p}}_n^m(0,\vc{\varpi}) - \vc{p}^m(0,\vc{\varpi})
\right\|
&\le& \sum_{h=0}^{m-1} 
\sum_{(l,j)\in\bbF} \left( \sum_{i\in\bbD} \varpi(i)\trunc{p}_n^h(0,i;l,j) \right)
\widetilde{\Delta}_n(l,j). \qquad
\label{add-eqn-17}
\end{eqnarray}

We now define $\trunc{\widetilde{\vc{P}}}_n$ as the
last-column-block-augmented first-$n$-block-column truncation of
$\widetilde{\vc{P}}$ and $\trunc{\widetilde{\vc{\pi}}}_n =
(\trunc{\widetilde{\pi}}_n(k,i))_{(n,i)\in\bbF}$ as the stationary
distribution of $\trunc{\widetilde{\vc{P}}}_n$.  We also define
$\trunc{\widetilde{\vc{p}}}_n^m(k,i)
=(\trunc{\widetilde{p}}_n^m(k,i;l,j))_{(l,j)\in\bbF}$ as a probability
vector such that $\trunc{\widetilde{p}}_n^m(k,i;l,j)$ represents the
$(k,i;l,j)$th element of $(\trunc{\widetilde{\vc{P}}}_n)^m$.  It then
follows from $\trunc{\vc{P}}_n \prec_d \trunc{\widetilde{\vc{P}}}_n$
and Proposition~\ref{prop-4}~(b) that $(\trunc{\vc{P}}_n)^h \prec_d
(\trunc{\widetilde{\vc{P}}}_n)^h$ for $h \in \bbN$. Therefore
Remark~\ref{prop-3} and $(\widetilde{\Delta}_n(l,j))_{(l,j)\in\bbF}
\in \sfBI_d$ (due to $\widetilde{\vc{P}} \in \sfBM_d$) yield
\begin{eqnarray}
\sum_{(l,j)\in\bbF} \trunc{p}_n^h(0,i;l,j) 
\widetilde{\Delta}_n(l,j)
&\le& \sum_{(l,j)\in\bbF} \trunc{\widetilde{p}}_n^h(0,i;l,j)
\widetilde{\Delta}_n(l,j).
\label{add-eqn-18}
\end{eqnarray}
Substituting (\ref{add-eqn-18}) into (\ref{add-eqn-17}), we have
\begin{eqnarray*}
\left\|
\trunc{\vc{p}}_n^m(0,\vc{\varpi}) - \vc{p}^m(0,\vc{\varpi})
\right\|
&\le& \sum_{h=0}^{m-1} 
\sum_{(l,j)\in\bbF} \left( \sum_{i\in\bbD} \varpi(i)\trunc{\widetilde{p}}_n^h(0,i;l,j) \right)
\widetilde{\Delta}_n(l,j).
\end{eqnarray*}
In addition, since $\trunc{\widetilde{\vc{P}}}_n \prec_d
\widetilde{\vc{P}} \in \sfBM_d$, Proposition~\ref{prop-4}~(c) implies
that $\trunc{\widetilde{\vc{\pi}}}_n \prec_d \widetilde{\vc{\pi}}$ and
thus
$\sum_{k=0}^{\infty}\trunc{\widetilde{\pi}}_n(k,i)=\sum_{k=0}^{\infty}\widetilde{\pi}(k,i)$
for $i \in \bbD$. Combining this with (\ref{eqn-tilde{J}=J}), we have
$\varpi(i) = \sum_{k=0}^{\infty}\trunc{\widetilde{\pi}}_n(k,i)$ for $i
\in \bbD$.  As a result, according to the discussion following
(\ref{eqn-14}) in the proof of Theorem~\ref{thm-main-geo}, we can
prove that
\[
\left\|
\trunc{\vc{p}}_n^m(0,\vc{\varpi}) - \vc{p}^m(0,\vc{\varpi})
\right\|
\le 2m\sum_{i\in\bbD}\trunc{\widetilde{\pi}}_n(n,i)
\le
{2mb \over 1-\gamma}\sum_{i\in\bbD}{1 \over v(n,i)}.
\]
\end{proof}

We can relax (\ref{ineqn-hat{P}v}) if the direct path to the states
$\{(0,i);i\in\bbD\}$ is enough ``large".
\begin{thm}\label{thm-extended-geo2}
Suppose that conditions (i) and (ii) of Theorem~\ref{thm-extended-geo}
are satisfied. Further suppose that there exists a column vector
$\vc{v}'=(v'(k,i))_{(k,i)\in\bbF} \in \sfBI_d$ such that $\vc{v}' \ge
\vc{e}$ and for some $\gamma' \in (0,1)$, $b' \in (0, \infty)$ and $K
\in \bbZ_+$,
\begin{eqnarray}
\widetilde{\vc{P}}\vc{v}' 
&\le& \gamma' \vc{v}' + b'\dd{1}_K,
\label{ineqn-hat{P}v'}
\\
\widetilde{\vc{P}}(K;0)\vc{e} &>& \vc{0},
\label{ineqn-hat{P}(K,0)e}
\end{eqnarray}
where $\widetilde{\vc{P}}(k;l)$ ($k,l\in\bbZ_+$) denotes a $d \times
d$ matrix such that $\widetilde{\vc{P}}(k;l) =
(\widetilde{p}(k,i;l,j))_{(i,j)\in\bbD}$.
Under these conditions, (\ref{geo-bound-2}) holds for all $n \in
\bbN$, where
\begin{eqnarray}
\gamma &=& {\gamma' + B \over 1 + B};
\label{defn-lambda}
\\
b &=& b' + B;
\label{defn-b}
\\
v(k,i) &=& 
\left\{
\begin{array}{ll}
v'(0,i), & k=0,~i\in\bbD,
\\
v'(k,i) + B, & k\in\bbN,~i\in\bbD;~\mbox{and}
\end{array}
\right.
\label{defn-v'}
\\
B &\in& (0,\infty)~\mbox{such that}~ B\cdot\widetilde{\vc{P}}(K;0)\vc{e} \ge b'\vc{e}.
\label{defn-B}
\end{eqnarray}
\end{thm}

\begin{rem}
The condition (\ref{ineqn-hat{P}(K,0)e}) ensures that there exists
some $B\in (0,\infty)$ satisfying (\ref{defn-B}). Further since
$\widetilde{\vc{P}} \in \sfBM_d$, (\ref{ineqn-hat{P}(K,0)e}) implies
$\widetilde{\vc{P}}(k;0)\vc{e} > \vc{0}$ for all $k=0,1,\dots,K$.
\end{rem}

\begin{proof}[Proof of Theorem~\ref{thm-extended-geo2}]
According to Theorem~\ref{thm-extended-geo}, it suffices to prove that
(\ref{ineqn-hat{P}v}) holds for some $\gamma \in (0,1)$, $b \in (0,
\infty)$ and $\vc{v} \in \sfBI_d$ with $\vc{v} \ge \vc{e}$. Let
$\vc{v}(k)$ and $\vc{v}'(k)$ ($k\in\bbZ_+$) denote $d \times 1$
vectors such that $\vc{v}(k)=(v(k,i))_{i\in\bbD}$ and
$\vc{v}'(k)=(v'(k,i))_{i\in\bbD}$. Clearly,
$\vc{v}=(\vc{v}(0)^{\rmT},\vc{v}(1)^{\rmT},\dots)^{\rmT}$ and
$\vc{v}'=(\vc{v}'(0)^{\rmT},\vc{v}'(1)^{\rmT},\dots)^{\rmT}$, where
the superscript ``$\rmT$" represents the transpose operator. Thus
(\ref{ineqn-hat{P}v'}), (\ref{defn-b}) and (\ref{defn-v'}) yield
\begin{eqnarray}
\sum_{l=0}^{\infty}
\widetilde{\vc{P}}(0;l)\vc{v}(l)
&\le& \sum_{l=0}^{\infty}\widetilde{\vc{P}}(0;l)\vc{v}'(l) + B\vc{e}
\le \gamma'\vc{v}'(0) + (b'+B)\vc{e}
\nonumber
\\
&=& \gamma'\vc{v}(0) + b\vc{e}
\le \gamma\vc{v}(0) + b\vc{e},
\label{add-eqn-07}
\end{eqnarray}
where the last inequality follows from $\gamma \ge \gamma'$ (due to
(\ref{defn-lambda})).

Further since $\widetilde{\vc{P}} \in \sfBM_d$,
$\sum_{l=1}^{\infty}\widetilde{\vc{P}}(k;l) \le
\sum_{l=1}^{\infty}\widetilde{\vc{P}}(K;l)$ for $k=1,2,\dots,K$. From
this and (\ref{defn-v'}), we have for $k=1,2,\dots,K$,
\begin{eqnarray}
\sum_{l=0}^{\infty}
\widetilde{\vc{P}}(k;l)\vc{v}(l)
&\le& \sum_{l=0}^{\infty}\widetilde{\vc{P}}(k;l)\vc{v}'(l)
 + B\sum_{l=1}^{\infty}\widetilde{\vc{P}}(K;l)\vc{e}
\nonumber
\\
&=& \sum_{l=0}^{\infty}\widetilde{\vc{P}}(k;l)\vc{v}'(l)
 + B\{\vc{e} - \widetilde{\vc{P}}(K;0)\vc{e}\}.
\label{add-eqn-10}
\end{eqnarray}
Applying (\ref{ineqn-hat{P}v'}) and (\ref{defn-B}) to the right hand
side of (\ref{add-eqn-10}), we obtain for $k=1,2,\dots,K$,
\begin{eqnarray}
\sum_{l=0}^{\infty}
\widetilde{\vc{P}}(k;l)\vc{v}(l)
&\le& \gamma'\vc{v}'(k) + B\vc{e} 
+ \{ b'\vc{e} - B\widetilde{\vc{P}}(K;0)\vc{e} \}
\le \gamma'\vc{v}'(k) + B\vc{e}.
\label{add-eqn-06}
\end{eqnarray}
Note here that (\ref{defn-lambda}) implies $\sup_{x\ge1}(\gamma'
x+B)/(x+B) = \gamma$. Thus since $\vc{v}' \ge \vc{e}$, we have
$\gamma' v'(k,i) + B \le \gamma (v'(k,i) + B)$.  Combining this with
(\ref{defn-v'}) yields
\begin{equation}
\gamma'\vc{v}'(k) + B\vc{e}
\le \gamma(\vc{v}'(k) + B\vc{e}) 
= \gamma \vc{v}(k),
\qquad k \in \bbN.
\label{ineqn-lambda-v(k)}
\end{equation}
Substituting (\ref{ineqn-lambda-v(k)}) into (\ref{add-eqn-06}), we
have
\begin{equation}
\sum_{l=0}^{\infty} \widetilde{\vc{P}}(k;l)\vc{v}(l)
\le \gamma \vc{v}(k), \qquad k=1,2,\dots,K.
\label{add-eqn-08}
\end{equation}
Similarly, for $k=K+1,K+2,\dots$,
\begin{eqnarray}
\sum_{l=0}^{\infty}
\widetilde{\vc{P}}(k;l)\vc{v}(l)
&\le& \sum_{l=0}^{\infty}\widetilde{\vc{P}}(k;l)\vc{v}'(l) + B\vc{e}
\le \gamma'\vc{v}'(k) + B\vc{e}
\le  \gamma\vc{v}(k),
\label{add-eqn-09}
\end{eqnarray}
where the last inequality is due to (\ref{ineqn-lambda-v(k)}). 
Finally, (\ref{add-eqn-07}),
(\ref{add-eqn-08}) and (\ref{add-eqn-09}) yield (\ref{ineqn-hat{P}v}).
\end{proof}

\section{Applications}\label{sec-applications}

In this section, we discuss the application of our results to
GI/G/1-type Markov chains. To this end, we make the following
assumption.
\begin{assumpt}\label{assumpt-MMRW-light}
(i) $\vc{P}$ is of the following form:
\begin{equation}
\vc{P}
=
\left(
\begin{array}{ccccc}
\vc{B}(0)  & \vc{B}(1)  & \vc{B}(2)  & \vc{B}(3) & \cdots
\\
\vc{B}(-1) & \vc{A}(0)  & \vc{A}(1)  & \vc{A}(2) & \cdots
\\
\vc{B}(-2) & \vc{A}(-1) & \vc{A}(0)  & \vc{A}(1) & \cdots
\\
\vc{B}(-3) & \vc{A}(-2) & \vc{A}(-1) & \vc{A}(0) & \cdots
\\
\vdots     & \vdots     &  \vdots    & \vdots    & \ddots
\end{array}
\right),
\label{GIG1-type-P}
\end{equation}
where $\vc{A}(k)$ and $\vc{B}(k)$ ($k=0,\pm1,\pm2,\dots$) are $d
\times d$ matrices; (ii) $\vc{P} \in \sfBM_d$; (iii) $\vc{P}$ is
irreducible and positive recurrent; (iv)
$\vc{A}:=\sum_{k=-\infty}^{\infty}\vc{A}(k)$ is irreducible and
stochastic; and (v) $r_{A_+} = \sup\{z > 0;
\mbox{$\sum_{k=0}^{\infty}z^k \vc{A}(k)$ is finite}\} > 1$.
\end{assumpt}

It follows from conditions (i), (ii) and (iv) of
Assumption~\ref{assumpt-MMRW-light} and
Proposition~\ref{prop-gamma(i,j)} that $\vc{\varPsi} =
\sum_{l=0}^{\infty}\vc{B}(k) = \vc{B}(-k) +
\sum_{l=-k+1}^{\infty}\vc{A}(l)$ for all $k \in \bbN$, which implies
that $\lim_{k\to\infty}\vc{B}(-k) = \vc{O}$ and thus $\vc{A} =
\vc{\varPsi}$.

Let $\widehat{\vc{A}}(z)$ denote
\begin{equation}
\widehat{\vc{A}}(z)
=
\sum_{k=-\infty}^{\infty}z^k \vc{A}(k),
\qquad z \in (1/r_{A_-},r_{A_+}) \cap \{1\} =: \calI_A,
\label{defn-tilde{A}(z)}
\end{equation}
where $r_{A_-} = \sup\{z > 0; \mbox{$\sum_{k=1}^{\infty}z^k
  \vc{A}(-k)$ is finite}\} \ge 1$. Let $\delta_A(z)$ ($z \in \calI_A$)
denote the real and maximum-modulus eigenvalue of
$\widehat{\vc{A}}(z)$ (see, e.g., Theorems 8.3.1 and 8.4.4 in
\cite{Horn90}). Let $\vc{\mu}_A(z)=(\mu_A(z,i))_{i\in\bbD}$ and
$\vc{v}_A(z)=(v_A(z,i))_{i\in\bbD}$ ($z \in \calI_A$) denote left- and
right-eigenvectors of $\widehat{\vc{A}}(z)$ corresponding to
eigenvalue $\delta_A(z)$, i.e.,
\begin{equation}
\vc{\mu}_A(z)
\widehat{\vc{A}}(z)
= \delta_A(z)\vc{\mu}_A(z),
\quad 
\widehat{\vc{A}}(z)\vc{v}_A(z)
= \delta_A(z)\vc{v}_A(z),
\label{add-eqn-03}
\end{equation}
which are normalized such that $\vc{\mu}_A(z)\vc{v}_A(z) = 1$ and
$\vc{v}_A(z) \ge \vc{e}$ for $z \in \calI_A$. We then have
$\delta_A(z) = \vc{\mu}_A(z)\widehat{\vc{A}}(z)\vc{v}_A(z)$. It also
follows from $\vc{A} = \vc{\varPsi}$ and condition (iv) of
Assumption~\ref{assumpt-MMRW-light} that $\delta_A(1)= 1$,
$\vc{\mu}_A(1) = c \vc{\varpi}$ and $\vc{v}_A(1) = c^{-1} \vc{e}$ for
some $c \in (0,1]$.

\begin{lem}\label{add-lem01}
Under Assumption~\ref{assumpt-MMRW-light}, there exists an $\alpha \in
(1,r_A)$ such that $\delta_A(\alpha) < 1$.
\end{lem}

\begin{proof}
Since $\delta_A(1)= 1$ and $\delta_A(z)$ is differentiable for $z \in
\calI_A$ (see Theorem~2.1 in \cite{Andr93}), it suffices to show that
$\delta_A'(1) < 0$. Indeed, $\delta_A'(1) =
\vc{\mu}_A(1)\sum_{k=-\infty}^{\infty}k\vc{A}(k)\vc{v}_A(1) =
\vc{\varpi}\sum_{k=-\infty}^{\infty}k\vc{A}(k)\vc{e}$, which is equal
to the mean drift of the process $\{X_{\nu};\nu\in\bbZ_+\}$ away from
the boundary and is strictly negative under
Assumption~\ref{assumpt-MMRW-light} (see, e.g., Proposition~2.2.1 in
\cite{Kimu13}).
\end{proof}

We now define $\vc{P}(k;l)$ ($k,l\in\bbZ_+$) as a $d \times d$ matrix
such that $\vc{P}(k;l)=(p(k,i;l,j))_{i,j\in\bbD}$. We also fix
$\vc{v}'=(\vc{v}'(0)^{\rmT},\vc{v}'(1)^{\rmT},\dots)^{\rmT}$ such that
\begin{eqnarray}
\vc{v}'(k) 
&=& \alpha^k \vc{v}_A(\alpha),\qquad k \in \bbZ_+,
\label{eqn-v'(k)}
\end{eqnarray}
which leads to $\vc{v}' \in \sfBI_d$. 
It then follows from (\ref{GIG1-type-P}) and (\ref{eqn-v'(k)}) that
\begin{eqnarray}
\sum_{l=0}^{\infty}\vc{P}(0;l)\vc{v}'(l)
&=& \sum_{l=0}^{\infty}\alpha^l\vc{B}(l) \cdot \vc{v}_A(\alpha)
=: \vc{w}(0),
\label{add-eqn-11'}
\\
\sum_{l=0}^{\infty}\vc{P}(k;l)\vc{v}'(l)
&=& \vc{B}(-k) \vc{v}_A(\alpha)
+ \alpha^k \sum_{l=-k+1}^{\infty}\alpha^l\vc{A}(l) \cdot \vc{v}_A(\alpha)
=: \vc{w}(k),
\quad k \in \bbN, \qquad
\label{add-eqn-11}
\end{eqnarray}
where $\vc{w}(0) \le \vc{w}(0) \le \vc{w}(1) \le \cdots$ because
$\vc{P} \in \sfBM_d$ and $\vc{v}' \in \sfBI_d$ (see
Proposition~\ref{prop-2}). Further, using (\ref{defn-tilde{A}(z)}) and
(\ref{add-eqn-03}), we can estimate the right hand side of
(\ref{add-eqn-11}) as follows:
\begin{eqnarray}
\sum_{l=0}^{\infty}\vc{P}(k;l)\vc{v}'(l)
= \vc{w}(k)
&\le& \vc{B}(-k) \vc{v}_A(\alpha) 
+ \alpha^k \widehat{\vc{A}}(\alpha)\vc{v}_A(\alpha)
\nonumber
\\
&=& \vc{B}(-k) \vc{v}_A(\alpha) 
+ \alpha^k \delta_A(\alpha) \vc{v}_A(\alpha) < \infty, \qquad
k \in \bbN,
\label{add-eqn-15}
\end{eqnarray}
which shows that $\vc{w}(k)$ is finite for all $k \in \bbZ_+$.
Combining (\ref{add-eqn-15}), $\lim_{k\to\infty}\vc{B}(-k) = \vc{O}$,
$\vc{v}_A(\alpha) \ge \vc{e}$ and Lemma~\ref{add-lem01}, we can show
that there exist some $\gamma' \in (0,1)$ and $k_{\ast} \in \bbN$ such
that
\begin{equation}
\sum_{l=0}^{\infty}\vc{P}(k;l)\vc{v}'(l)
\le \gamma' \alpha^k \vc{v}_A(\alpha) = \gamma' \vc{v}'(k),
\qquad \forall k \ge k_{\ast},
\label{add-eqn-12}
\end{equation}
where the last equality is due to (\ref{eqn-v'(k)}).

\begin{thm}
Suppose that Assumption~\ref{assumpt-MMRW-light} holds and fix
$\gamma' \in (0,1)$ and $k_{\ast} \in \bbN$ satisfying
(\ref{add-eqn-12}).  Further if $\vc{B}(-K)\vc{e} > \vc{0}$ for some
nonnegative integer $K \ge k_{\ast} - 1$, then the bound
(\ref{geo-bound-2}) holds for $\gamma \in (0,1)$, $b \in (0,\infty)$
and $\vc{v} \in \sfBI_d$ such that (\ref{defn-lambda})--(\ref{defn-B})
are satisfied, where $\vc{v}'$ is given by (\ref{eqn-v'(k)}),
$\widetilde{\vc{P}}(K;0)= \vc{B}(-K)$ and
\begin{eqnarray}
b' &=& \inf\{
x > 0; x\vc{e} \ge \vc{w}(k)- \gamma'\alpha^k \vc{v}_A(\alpha)~
(0 \le \forall  k \le K) 
\}.
\label{eqn-b'}
\end{eqnarray}
\end{thm}

\begin{proof}
Fix $\widetilde{\vc{P}} = \vc{P} \in \sfBM_d$.  From
(\ref{eqn-v'(k)})--(\ref{add-eqn-11}) and
(\ref{eqn-b'}), we then have
\[
\sum_{l=0}^{\infty}\widetilde{\vc{P}}(k;l)\vc{v}'(l)
= \gamma'\vc{v}'(k) + \{ \vc{w}(k) - \gamma'\alpha^k \vc{v}_A(\alpha) \}
\le \gamma'\vc{v}'(k) + b'\vc{e},~~~k=0,1,\dots,K.
\]
This inequality and (\ref{add-eqn-12}) yield
(\ref{ineqn-hat{P}v'}). Further (\ref{ineqn-hat{P}(K,0)e}) holds
because $\widetilde{\vc{P}}(K;0)\vc{e}= \vc{B}(-K)\vc{e} > \vc{0}$. As
a result, all the conditions of Theorem~\ref{thm-extended-geo2} are
satisfied and thus the bound (\ref{geo-bound-2}) holds.
\end{proof}

Finally, we consider a special case where $\vc{B}(-k) = \vc{A}(-k) =
\vc{O}$ for $k \ge 2$, $\vc{B}(-1) = \vc{A}(-1)$ and $\vc{B}(k) =
\vc{A}(k-1)$ for $k \in \bbZ_+$, i.e.,
\begin{equation}
\vc{P}
=
\left(
\begin{array}{ccccc}
\vc{A}(-1) & \vc{A}(0) & \vc{A}(1) & \vc{A}(2) & \cdots
\\
\vc{A}(-1) & \vc{A}(0) & \vc{A}(1) & \vc{A}(2) & \cdots
\\
\vc{O} & \vc{A}(-1) & \vc{A}(0) & \vc{A}(1) & \cdots
\\
\vc{O} & \vc{O} & \vc{A}(-1) & \vc{A}(0) & \cdots
\\
\vdots & \vdots & \vdots & \vdots & \ddots
\end{array}
\right),
\label{MG1-type-P}
\end{equation}
which is block-monotone with block size $d$.  Note that $\vc{P}$ in
(\ref{MG1-type-P}) is an M/G/1-type transition probability matrix and
appears in the analysis of the stationary queue length distribution in
the BMAP/GI/1 queue (see \cite{Taki00}).
\begin{thm}
Suppose that Assumption~\ref{assumpt-MMRW-light} holds. Further if
$\vc{B}(-k) = \vc{A}(-k) = \vc{O}$ for $k \ge 2$, $\vc{B}(-1) =
\vc{A}(-1)$ and $\vc{B}(k) = \vc{A}(k-1)$ for $k \in \bbZ_+$, then the
bound (\ref{geo-bound-2}) holds for $\gamma = \delta_A(\alpha)$,
$b=(\alpha-1)\max_{i\in\bbD}v_A(\alpha,i)$ and $\vc{v}=\vc{v}'$ given
in (\ref{eqn-v'(k)}).
\end{thm}

\begin{proof}
Fixing $\vc{v}=\vc{v}'$ and applying
(\ref{defn-tilde{A}(z)})--(\ref{eqn-v'(k)}), Lemma~\ref{add-lem01} and
the conditions on $\{\vc{B}(k)\}$ to (\ref{add-eqn-11'}) and
(\ref{add-eqn-11}), we obtain
\begin{eqnarray*}
\sum_{l=0}^{\infty}\vc{P}(0;l)\vc{v}(l)
&=& \alpha \delta_A(\alpha)\vc{v}_A(\alpha)
\le \vc{v}(0) + (\alpha - 1)\vc{v}_A(\alpha),
\\
\sum_{l=0}^{\infty}\vc{P}(k;l)\vc{v}(l)
&=& \alpha^k \delta_A(\alpha)\vc{v}_A(\alpha)
= \delta_A(\alpha)\vc{v}(k),
\quad k \in \bbN, \qquad
\end{eqnarray*}
which imply that all the conditions of Theorem~\ref{thm-main-geo}
hold. Thus we have (\ref{geo-bound-2}).
\end{proof}

\appendix

\section{Pathwise ordering}\label{sec-pathwise-ordering}

This section presents two lemmas on the pathwise ordering associated
with BMMCs. As in the previous sections, we use
$\vc{P}=(p(k,i;l,j))_{(k,i),(l,j)\in\bbF}$ and
$\widetilde{\vc{P}}=(\widetilde{p}(k,i;l,j))_{(k,i),(l,j)\in\bbF}$ to
represent $|\bbF| \times |\bbF|$ stochastic matrices, though they are
not necessarily assumed to be irreducible or recurrent in this
section.

Let $\{U_{\nu};\nu\in\bbN\}$ and $\{S_{\nu};\nu\in\bbN\}$ denote two
independent sequences of independent and identically distributed
(i.i.d.)\ random variables on a probability space $(\Omega,\calF,\PP)$
such that $U_{\nu}$ and $S_{\nu}$ are uniformly distributed in
$(0,1)$. Let $J_0^{\ast}$ denote a $\bbD$-valued random variable on
the probability space $(\Omega,\calF,\PP)$, which is independent of
both $\{U_{\nu};\nu\in\bbN\}$ and $\{S_{\nu};\nu\in\bbN\}$. Further
let $J_{\nu}^{\ast} = G^{-1}(S_{\nu} \mid J_{\nu-1}^{\ast})$ for
$\nu\in\bbN$, where
\[
G^{-1}(s \mid i) 
= \inf\left\{
j \in\bbD; \sum_{j'=1}^j \psi(i,j')  \ge s
\right\},\qquad  0 < s < 1,~i\in\bbD.
\]
It then follows that $\{J^*_{\nu};\nu\in\bbZ_+\}$ is a $\bbD$-valued
Markov chain on the probability space $(\Omega,\calF,\PP)$ such that
$\PP(J^*_{\nu+1} = j \mid J^*_{\nu} = i) = \psi(i,j)$ for $i,j\in\bbD$
and $\nu \in \bbZ_+$, where $\psi(i,j)$ is defined in
Proposition~\ref{prop-gamma(i,j)}.

\begin{lem}[Pathwise ordered property of BMMCs]\label{lem1-discrete-ordering}
Suppose $\vc{P} \in \sfBM_d$. Let $X'_0$ and $X''_0$ denote
nonnegative integer-valued random variables on the probability space
$(\Omega,\calF,\PP)$, which are independent of both
$\{U_{\nu};\nu\in\bbN\}$ and $\{S_{\nu};\nu\in\bbN\}$. Further let
$X'_{\nu} = F^{-1}(U_{\nu} \mid X'_{\nu-1},J^*_{\nu-1},J^*_{\nu})$ and
$X''_{\nu} = F^{-1}(U_{\nu} \mid X''_{\nu-1},J^*_{\nu-1},J^*_{\nu})$
for $\nu\in\bbN$, where $F^{-1}(u \mid k,i,j)$ ($0<u<1$,
$k\in\bbZ_+,i,j\in\bbD$) is defined as
\begin{equation}
F^{-1}(u \mid k,i,j) 
= \inf\left\{
l \in\bbZ_+; \sum_{m=0}^l {p(k,i;m,j) \over \psi(i,j)}\ge u
\right\}.
\label{defn-F^{-1}}
\end{equation}
Under these conditions, $\{(X'_{\nu},J^*_{\nu});\nu\in\bbZ_+\}$ and
$\{(X''_{\nu},J^*_{\nu});\nu\in\bbZ_+\}$ are Markov chains with
transition probability matrix $\vc{P}$ on the probability space
$(\Omega,\calF,\PP)$ such that $X'_{\nu} \le X''_{\nu}$ for all $\nu
\in \bbN$ if $X'_0 \le X''_0$.
\end{lem}

\begin{proof}
Suppose that $X'_{\nu} \le X''_{\nu}$ for some $\nu \in
\bbZ_+$.  It then follows from $\vc{P} \in \sfBM_d$ that
\[
\sum_{m=0}^l p(X'_{\nu},J^*_{\nu};m,J^*_{\nu+1}) 
\ge \sum_{m=0}^l p(X''_{\nu},J^*_{\nu};m,J^*_{\nu+1}),
\qquad l\in\bbZ_+.
\]
Thus from the definition of $\{X'_{\nu}\}$ and $\{X''_{\nu}\}$, we
have
\begin{eqnarray*}
X''_{\nu+1} 
&=& \inf\left\{
l \in\bbZ_+; \sum_{m=0}^l 
{ p(X''_{\nu},J^*_{\nu};m,J^*_{\nu+1}) \over \psi(J^*_{\nu},J^*_{\nu+1}) } 
\ge U_{\nu+1}
\right\}
\nonumber
\\
&\ge& \inf\left\{
l \in\bbZ_+; \sum_{m=0}^l 
{ p(X'_{\nu},J^*_{\nu};m,J^*_{\nu+1}) \over \psi(J^*_{\nu},J^*_{\nu+1}) } 
\ge U_{\nu+1}
\right\}
\nonumber
\\
&=& F^{-1}(U_{\nu+1} \mid X'_{\nu},J^*_{\nu},J^*_{\nu+1}) = X'_{\nu+1}.
\end{eqnarray*}
Therefore it is proved by induction that $X'_{\nu} \le X''_{\nu}$ for
all $\nu \in \bbN$.

Next we prove that the dynamics of
$\{(X'_{\nu},J^*_{\nu});\nu\in\bbZ_+\}$ is determined by $\vc{P}$. Let
$\sigma(\,\cdot\,)$ denote the sigma-algebra generated by the random
variables in the parentheses. From the definition of
$\{(X'_{\nu},J^*_{\nu})\}$, we then have for $\nu \in \bbN$,
\begin{eqnarray}
\lefteqn{
\sigma(X'_0,X'_1,\dots,X'_{\nu-1},J_0^{\ast},J_1^{\ast},\dots,J_{\nu-1}^{\ast})
}
\quad &&
\nonumber
\\
&& {} \subseteq 
\sigma(X'_0, J^*_0, U_1,U_2,\dots,U_{\nu-1},S_1,S_2,\dots,S_{\nu-1})
=:\calG_{\nu-1}.
\end{eqnarray}
Note here that for $(k,i) \in \bbF$ and $j\in\bbD$,
\[
\calG_{\nu-1} \cap 
\{X'_{\nu}=k, J^*_{\nu}=i, J^*_{\nu+1}=j\}
\subseteq 
\sigma(X'_0, J^*_0,U_1,U_2,\dots,U_{\nu},S_1,S_2,\dots,S_{\nu+1}),
\]
which implies that $U_{\nu+1}$ is independent of both $\calG_{\nu-1}$
and $\{X'_{\nu}=k, J^*_{\nu}=i, J^*_{\nu+1}=j\}$ for $(k,i) \in \bbF$
and $j\in\bbD$.  Thus it follows from the definition of $\{X'_{\nu}\}$
that
\begin{eqnarray}
\lefteqn{
\PP(X'_{\nu+1} \le l \mid \calG_{\nu-1}, X'_{\nu}=k, J^*_{\nu}=i, J^*_{\nu+1}=j)
}
\qquad &&
\nonumber
\\
&=& \PP\left( \left.
\sum_{m=0}^l {p(k,i;m,j) \over \psi(i,j)} \ge U_{\nu+1}\,
\right| \calG_{\nu-1}, X'_{\nu}=k, J^*_{\nu}=i, J^*_{\nu+1}=j 
\right)
\nonumber
\\
&=& \PP\left(
\sum_{m=0}^l {p(k,i;m,j) \over \psi(i,j)} \ge U_{\nu+1}
\right)
= \sum_{m=0}^l {p(k,i;m,j) \over \psi(i,j)},
\quad (k,i) \times (l,j) \in~\bbF^2. \qquad~
\label{add-eqn-20}
\end{eqnarray}
Note also that $S_{\nu+1}$ is independent of $\calG_{\nu} \supseteq
\calG_{\nu-1} \cap \{X'_{\nu}=k, J^*_{\nu}=i\}$ for $(k,i) \in
\bbF$. Therefore from the definition of $\{J^*_{\nu}\}$, we have for
$(k,i) \in \bbF$ and $j\in\bbD$,
\begin{eqnarray}
\lefteqn{
\PP(J^*_{\nu+1} = j \mid \calG_{\nu-1}, X'_{\nu}=k, J^*_{\nu}=i) 
}
\qquad &&
\nonumber
\\
&& {}
= \PP\left(
\left.
\sum_{j'=1}^{j-1} \psi(i,j') < S_{\nu+1} \le \sum_{j'=1}^{j} \psi(i,j') 
\,\right|
\calG_{\nu-1}, X'_{\nu}=k, J^*_{\nu}=i
\right)
\nonumber
\\
&& {}
= \PP\left(
\sum_{j'=1}^{j-1} \psi(i,j') < S_{\nu+1} \le \sum_{j'=1}^{j} \psi(i,j')
\right)
= \psi(i,j).
\label{add-eqn-21}
\end{eqnarray}
Combining (\ref{add-eqn-20}) and (\ref{add-eqn-21}) yields
\begin{eqnarray*}
\lefteqn{
\PP(X'_{\nu+1} \le l, J^*_{\nu+1}=j \mid \calG_{\nu-1}, X'_{\nu}=k, J^*_{\nu}=i)
}
\quad &&
\nonumber
\\
&=& \PP(X'_{\nu+1} \le l \mid \calG_{\nu-1}, X'_{\nu}=k, J^*_{\nu}=i, J^*_{\nu+1}=j)
\PP(J^*_{\nu+1}=j \mid \calG_{\nu-1}, X'_{\nu}=k, J^*_{\nu}=i)
\nonumber
\\
&=& \sum_{m=0}^l p(k,i;m,j),
\qquad (k,i) \times (l,j) \in \bbF^2,
\end{eqnarray*}
which shows that $\{(X'_{\nu},J^*_{\nu});\nu\in\bbZ_+\}$ is a Markov
chain with transition probability matrix $\vc{P}$ on the probability
space $(\Omega,\calF,\PP)$. The same argument holds for
$\{(X''_{\nu},J^*_{\nu});\nu\in\bbZ_+\}$. We omit the details.
\end{proof}

\begin{lem}[Pathwise ordering by the block-wise dominance]
\label{lem2-discrete-ordering}
Suppose $\vc{P} \prec_{d} \widetilde{\vc{P}}$ and either $\vc{P} \in
\sfBM_d$ or $\widetilde{\vc{P}} \in \sfBM_d$. Let $X_0^{\ast}$ and
$\widetilde{X}_0^{\ast}$ denote nonnegative integer-valued random
variables on the probability space $(\Omega,\calF,\PP)$, which are
independent of both $\{U_{\nu};\nu\in\bbN\}$ and
$\{S_{\nu};\nu\in\bbN\}$. Further let $X_{\nu}^{\ast} = F^{-1}(U_{\nu}
\mid X_{\nu-1}^{\ast},J^*_{\nu-1},J^*_{\nu})$ and
$\widetilde{X}_{\nu}^{\ast} = \widetilde{F}^{-1}(U_{\nu} \mid
\widetilde{X}_{\nu-1}^{\ast},J^*_{\nu-1},J^*_{\nu})$ for $\nu \in
\bbN$, where $F^{-1}(u \mid k,i,j)$ ($0<u<1$, $k\in\bbZ_+,i,j\in\bbD$)
is defined in (\ref{defn-F^{-1}}) and $\widetilde{F}^{-1}(u \mid
k,i,j)$ ($0<u<1$, $k\in\bbZ_+,i,j\in\bbD$) is defined as
\[
\widetilde{F}^{-1}(u \mid k,i,j) 
= \inf\left\{
l \in\bbZ_+; \sum_{m=0}^l {\widetilde{p}(k,i;m,j) \over \psi(i,j)}\ge u
\right\}.
\]
Under these conditions,
$\{(X_{\nu}^{\ast},J_{\nu}^{\ast});\nu\in\bbZ_+\}$ and
$\{(\widetilde{X}_{\nu}^{\ast},J_{\nu}^{\ast});\nu\in\bbZ_+\}$ are
Markov chains with transition probability matrices $\vc{P}$ and
$\widetilde{\vc{P}}$, respectively, on the probability space
$(\Omega,\calF,\PP)$ such that $X_{\nu}^{\ast} \le
\widetilde{X}_{\nu}^{\ast}$ for all $\nu \in \bbN$ if $X_0^{\ast} \le
\widetilde{X}_0^{\ast}$.
\end{lem}

\begin{proof}
Proposition \ref{prop-4}~(a) shows that for all $k \in \bbZ_+$ and
$i,j \in \bbD$,
\[
\psi(i,j)
=
\sum_{l=0}^{\infty}p(k,i;l,j) 
= \sum_{l=0}^{\infty}\widetilde{p}(k,i;l,j).
\]
Therefore, following the proof of Lemma~\ref{lem1-discrete-ordering},
we can prove that $\{(X_{\nu}^{\ast},J_{\nu}^{\ast});\nu\in\bbZ_+\}$
and $\{(\widetilde{X}_{\nu}^{\ast},J_{\nu}^{\ast});\nu\in\bbZ_+\}$ are
Markov chains with transition probability matrices $\vc{P}$ and
$\widetilde{\vc{P}}$, respectively, on the probability space
$(\Omega,\calF,\PP)$. Similarly we can prove by induction that if
$X_0^{\ast} \le \widetilde{X}_0^{\ast}$, then $X_{\nu}^{\ast} \le
\widetilde{X}_{\nu}^{\ast}$ for all $\nu \in \bbN$. We omit the details.
\end{proof}

\section*{Acknowledgments}
The author thanks an anonymous referee for his/her constructive
comments and suggestions on improving the presentation of this paper.
Research of the author was supported in part by Grant-in-Aid for Young
Scientists (B) of Japan Society for the Promotion of Science under
Grant No.~24710165.

\end{document}